\documentclass[12pt]{amsart}
\usepackage{amssymb,amsmath}
\usepackage{graphicx}
\usepackage{enumitem}
\usepackage{tikz}

\makeatletter
\@namedef{subjclassname@2020}{%
  \textup{2020} Mathematics Subject Classification}
\makeatother

\newcommand{\abs}[1]{\left|\,#1\,\right|}

\newcommand{\bQ}{\mathbb{Q}}
\newcommand{\bR}{\mathbb{R}}
\newcommand{\bZ}{\mathbb{Z}}
\newcommand{\cC}{{\mathcal{C}}}
\newcommand{\cS}{{\mathcal{S}}}
\newcommand{\cX}{{\mathcal{X}}}
\newcommand{\disp}{\displaystyle}
\newcommand{\et}{\quad\mbox{and}\quad}
\newcommand{\GL}{\mathrm{GL}}

\newcommand{\norm}[1]{\|\hspace*{1pt}#1\hspace*{1pt}\|}

\newcommand{\ssi}{\quad\Longleftrightarrow\quad}

\newcommand{\tdelta}{{\widetilde{\delta}}}

\newcommand{\tP}{{\widetilde{P}}}
\newcommand{\tuP}{{\widetilde{\uP}}}
\newcommand{\tuv}{{\widetilde{\uv}}}
\newcommand{\trT}{{}^tT}
\newcommand{\tuu}{{\widetilde{\uu}}}
\newcommand{\tux}{{\widetilde{\ux}}}

\newcommand{\tV}{{\widetilde{V}}}
\newcommand{\ua}{\mathbf{a}}
\newcommand{\ub}{\mathbf{b}}
\newcommand{\uc}{\mathbf{c}}
\newcommand{\ud}{\mathbf{d}}
\newcommand{\ue}{\mathbf{e}}
\newcommand{\uL}{\mathbf{L}}
\newcommand{\uP}{\mathbf{P}}
\newcommand{\uu}{\mathbf{u}}
\newcommand{\uv}{\mathbf{v}}
\newcommand{\uuv}{\underline{\uv}}
\newcommand{\uw}{\mathbf{w}}
\newcommand{\ux}{\mathbf{x}}
\newcommand{\uux}{\underline{\ux}}
\newcommand{\uy}{\mathbf{y}}
\newcommand{\uz}{\mathbf{z}}
\newcommand{\ualpha}{{\boldsymbol{\alpha}}}

\newcommand{\vol}{\mathrm{vol}}
\newcommand{\xbar}{\overline{x}}

\DeclareMathOperator{\dist}{dist}
\DeclareMathOperator{\proj}{proj}

\newtheorem{theorem}{Theorem}[section]
\newtheorem{lemma}[theorem]{Lemma}
\newtheorem{proposition}[theorem]{Proposition}
\newtheorem{cor}[theorem]{Corollary}

\theoremstyle{remark}
\newtheorem{definition}[theorem]{Definition}

\numberwithin{equation}{section}
\allowdisplaybreaks

\begin{document}

\title[Diophantine approximation with constraints]
{Diophantine approximation with constraints}
\author{J\'er\'emy Champagne}
\address{
Department of Pure Mathematics\\
University of Waterloo\\
Waterloo, Ontario N2L 3G1, Canada}
\email{jchampagne@uwaterloo.ca}
\author{Damien Roy}
\address{
   D\'epartement de Math\'ematiques\\
   Universit\'e d'Ottawa\\
   150 Louis Pasteur\\
   Ottawa, Ontario K1N 6N5, Canada}
\email{droy@uottawa.ca}


\begin{abstract}
Following Schmidt, Thurnheer and Bugeaud-Kristensen, we study how Dirichlet's
theorem on linear forms needs to be modified when one requires that the vectors
of coefficients of the linear forms make a bounded acute angle with respect
to a fixed proper non-zero subspace $V$ of $\bR^n$.
Assuming that the point of $\bR^n$ that we are approximating has linearly independent
coordinates over $\bQ$, we obtain best possible exponents of approximation which
surprisingly depend only on the dimension of $V$.  Our estimates are derived
by reduction to a result of Thurnheer, while their optimality follows from a new
general construction in parametric geometry of numbers involving angular constraints.
\end{abstract}

\maketitle

\baselineskip=15.2pt

%
%

\section{Introduction}
\label{sec:intro}

Given an integer $n\ge 2$ and an arbitrary point $\uu\in\bR^n$ with linearly
independent coordinates over $\bQ$, we know that there are infinitely many
integer points $\ux\in\bZ^n$ for which
\begin{equation}
\label{intro:eq:Dirichlet}
 \abs{\ux\cdot\uu} \le c_1\norm{\ux}^{-(n-1)}
\end{equation}
for a constant $c_1>0$ that depends only on $\uu$.  Here the dot represents the
usual scalar product in $\bR^n$ and the norm is the associated Euclidean norm
$\norm{\ux}=(\ux\cdot\ux)^{1/2}$. This is a result of Dirichlet
\cite[Chapter 2, Corollary 1D]{Sc1980}, and it is best possible in the sense that there
are points $\uu\in\bR^n$ with linearly independent coordinates over $\bQ$ which
satisfy $\abs{\ux\cdot\uu} \ge c_2\norm{\ux}^{-(n-1)}$ for any non-zero
$\ux\in\bZ^n$, with another constant $c_2>0$.  The latter occurs for example
when the coordinates of $\uu$ form a basis of a real number field of degree $n$
\cite[Chapter 2, Theorem 4A]{Sc1980}.

Of course, any point $\ux\in\bZ^n$ of large norm that satisfies \eqref{intro:eq:Dirichlet}
makes a small angle with the maximal subspace $\uu^\perp$ of $\bR^n$ orthogonal
to $\uu$.  In the present paper, we study how the right hand side of
\eqref{intro:eq:Dirichlet} has to be modified when $\ux$ is requested to
make a small angle with respect to a fixed proper non-zero subspace $V$ of $\bR^n$.
This line of research was initiated in 1976 by W.~M.~Schmidt \cite{Sc1976} and followed
by several authors \cite{Th1990, BK2009, Mo2012, Ro2014, Ch2021}, in chronological order.
Our first main result is the following statement where $\dist(\ux,V)$ denotes the sine
of the angle between a non-zero point $\ux$ and a non-zero subspace $V$ in $\bR^n$.

\begin{theorem}
\label{intro:thm1}
Let $m,n$ be integers with $m\ge 1$ and $n\ge m+2$, and let $V$ be a subspace of
$\bR^n$ of dimension $m+1$.  Set
\begin{equation}
 \rho=\rho_m=\frac{m+\sqrt{m^2+4m}}{2}.
 \label{intro:thm1:eq:rho}
\end{equation}
\begin{itemize}[labelindent=10pt, leftmargin=*]
\item[$1)$] For each point $\uu$ of $\bR^n$ whose coordinates are linearly independent
over $\bQ$, and each pair of numbers $\delta>0$ and $\epsilon>0$, there exists a
non-zero point $\ux$ of $\bZ^n$ with
\begin{equation}
\label{intro:thm1:eq1}
 \dist(\ux,V)\le \delta \et \abs{\ux\cdot\uu}\le \epsilon\norm{\ux}^{-\rho}.
\end{equation}
\item[$2)$] Conversely, let $\psi\colon[1,\infty)\to(0,\infty)$ be any unbounded
monotonically increasing function and let $\delta=1/\max\{4n,\,24(n-m)\}$.
There exists a point $\uu$ of $\bR^n$ with linearly independent coordinates
over $\bQ$ such that at most finitely many non-zero points $\ux$ of $\bZ^n$ satisfy
\begin{equation}
\label{intro:thm1:eq2}
 \dist(\ux,V)\le\delta
 \et
 \abs{\ux\cdot\uu}\le \psi(\norm{\ux})^{-1}\norm{\ux}^{-\rho}.
\end{equation}
\end{itemize}
\end{theorem}

For fixed $\uu$ and $\delta$ as in part 1) of the theorem, we obtain infinitely
many non-zero points $\ux\in\bZ^n$ with $\dist(\ux,V)\le \delta$ and
$\abs{\ux\cdot\uu}\le \norm{\ux}^{-\rho}$ by letting $\epsilon$ tend to zero.
On the other hand, if we choose $\psi(t)=t^\eta$ for some $\eta>0$,
then part 2) provides  $\uu$ and $\delta$ for which only finitely many
non-zero points $\ux\in\bZ^n$ have $\dist(\ux,V)\le \delta$ and
$\abs{\ux\cdot\uu}\le \norm{\ux}^{-\rho-\eta}$.  Thus the exponent $\rho$ in
\eqref{intro:thm1:eq1} is best possible.  Since
\begin{equation}
\label{intro:ineq:rho}
 \rho-m = \frac{\rho}{\rho+1} \in (0,1),
\end{equation}
we have $m<\rho<m+1$, and so $\rho$ is strictly smaller than Dirichlet's exponent
$n-1$ in \eqref{intro:eq:Dirichlet}.

\begin{cor}
\label{intro:cor0}
The statement of Theorem \ref{intro:thm1} remains true if the condition
$\dist(\ux,V)\le \delta$ in \eqref{intro:thm1:eq1} and \eqref{intro:thm1:eq2}
is replaced by $\dist(\ux,W)\le\delta$ where $W=V\cap\uu^\perp$.
\end{cor}

In fact, as we will see, this provides an equivalent form of the theorem.  Note
that, when $V$ is defined over $\bQ$, we have $V\not\subseteq\uu^\perp$
and so $\dim(W)=m$, in the notation of the above corollary.

Theorem \ref{intro:thm1} extends results of several authors.  In the case where
$m=1$ and
\begin{equation}
\label{intro:choixV:eq1}
 V=\{(x_1,\dots,x_n)\in\bR^n\,;\, x_2=\cdots=x_n\},
\end{equation}
we will see that it admits the following consequence where
\begin{equation*}
\label{intro:ineq:gamma}
 \gamma = \rho_1=(1+\sqrt{5})/2\simeq 1.618
\end{equation*}
denotes the golden ratio.

\begin{cor}
\label{intro:cor1}
Let $n\ge 3$ be an integer.  For each point $\uu$ of $\bR^n$ with linearly
independent coordinates over $\bQ$ and for each $\epsilon>0$, there exists a point
$\ux=(x_1,\dots,x_n)$ of $\bZ^n$ with
\begin{equation}
 \label{intro:cor1:eq1}
 x_2,\dots,x_n>0  \et \abs{\ux\cdot\uu}\le \epsilon\norm{\ux}^{-\gamma}.
\end{equation}
On the other hand, for each unbounded monotonically increasing function
$\psi$ from $[1,\infty)$ to $(0,\infty)$, there exists a point $\uu$ of $\bR^n$ with linearly
 independent coordinates over $\bQ$ for which at most finitely many points
$\ux=(x_1,\dots,x_n)$ of $\bZ^n$ satisfy
\begin{equation}
\label{intro:cor1:eq2}
 x_2,\dots,x_n>0
 \et
 \abs{\ux\cdot\uu}\le \psi(\norm{\ux})^{-1}\norm{\ux}^{-\gamma}.
\end{equation}
\end{cor}

For $n=3$, the first part of the corollary is the original result of Schmidt
\cite[Theorem 1]{Sc1976} from 1976.
Later, in \cite[Section 5]{Sc1983}, Schmidt conjectured that, in that case,
one could replace $\gamma$ in \eqref{intro:cor1:eq1} by any number smaller than $2$.
This was disproved in 2012 by Moshchevitin who showed, by an ingenious construction
in \cite{Mo2012}, that it cannot be replaced by a number larger than the largest real
root of $x^4-2x^2-4x+1$ which is approximately $1.947$.  Motivated by this, the
second author proved in \cite[Corollary]{Ro2014} that $\rho=\gamma$ is best possible
in the wider context of Theorem \ref{intro:thm1}, part 2), for $m=1$ and $n=3$.
Earlier,  in \cite[Remark F]{Sc1976}, Schmidt
had observed that, for $n\ge 3$, the exponent $\gamma$ cannot be replaced by a number
larger than $2$.

In the general case, we will see that the choice of
\begin{equation}
\label{intro:choixV:eq2}
 V=\{(x_1,\dots,x_n)\in\bR^n\,;\, x_{m+2}=\cdots=x_n=0\},
\end{equation}
yields the following statement.

\begin{cor}
\label{intro:cor2}
Let $m,n$ be integers with $1\le m\le n-2$ and let $\rho=\rho_m$ as in Theorem
\ref{intro:thm1}.   For each point $\uu$ of $\bR^n$ with linearly independent
coordinates over $\bQ$ and each choice of $\delta, \epsilon>0$, there exists a non-zero
point $\ux=(x_1,\dots,x_n)$ of $\bZ^n$ which satisfies
\begin{equation}
 \label{intro:cor2:eq1}
 \max\{|x_{m+2}|,\dots,|x_n|\} \le \delta\max\{|x_2|,\dots,|x_n|\}
 \et
 \abs{\ux\cdot\uu}\le \epsilon\norm{\ux}^{-\rho}.
\end{equation}
On the other hand, for each unbounded monotonically increasing function
$\psi$ from $[1,\infty)$ to $(0,\infty)$,
there exists a point $\uu$ of $\bR^n$ with linearly independent
coordinates over $\bQ$ for which at most finitely many points
$\ux=(x_1,\dots,x_n)$ of $\bZ^n$ satisfy
\begin{equation}
 \label{intro:cor2:eq2}
 \max\{|x_1|,\dots,|x_{m+1}|\} = \max\{|x_1|,\dots,|x_n|\}
 \et
 \abs{\ux\cdot\uu}\le \psi(\norm{\ux})^{-1}\norm{\ux}^{-\rho}.
\end{equation}
\end{cor}

For $n=m+2$ and any $m\ge 1$, the first part of the corollary was established
by Thurnheer in 1990 \cite[Theorem 1 (b)]{Th1990}, along with
other interesting results.  For $n=3$ and $m=1$, this result is equivalent to that
of Schmidt mentioned above.  At the level of exponents, we deduce the following
statement.

\begin{cor}
\label{intro:cor3}
Let $m,n$ be integers with $1\le m\le n-2$.  For each point $\uu\in\bR^n$ with
$\bQ$-linearly independent coordinates, let $\rho(\uu)$ denote the supremum of all
$\rho$ for which the inequality $\abs{\ux\cdot\uu} \le \norm{\ux}^{-\rho}$ has
infinitely many solutions $\ux=(x_1,\dots,x_n)\in\bZ^n$ with
\begin{equation}
 \label{intro:cor3:eq}
 \max\{|x_{m+2}|,\dots,|x_n|\} < \max\{|x_2|,\dots,|x_{m+1}|\}.
\end{equation}
Then, we have $\rho(\uu)\ge \rho_m$ for all those $\uu$, and the equality
$\rho(\uu)=\rho_m$ for uncountably many unit vectors $\uu$ among them.
\end{cor}

In \cite[Theorem 1]{BK2009}, Bugeaud and Kristensen showed that $\rho(\uu)=n-1$
for almost all $\uu$ with respect to the Lebesgue measure, and that each number $\rho\ge m+1$
is equal to $\rho(\uu)$ for uncountably many $\uu$ with first coordinate equal to $1$
($\rho(\uu)$ is a projective invariant).  They further asked two problems
about the existence of points $\uu$ with $\bQ$-linearly independent coordinates
satisfying $\rho(\uu)<m+1$.   Since $\rho_m<m+1$,
the above corollary shows that $\rho_m$ is the smallest such value, answering
positively their first problem and negatively the other.

The 2021 MSc thesis of the first author establishes Theorem \ref{intro:thm1} for $m=1$ and
any $n\ge 3$.  Part 1) is proved as \cite[Theorem 4.2.4]{Ch2021}, and part 2) as
\cite[Theorem 4.3.1]{Ch2021}.  For part 1), it is reasonable that the exponent
$\rho=\gamma$ which works for $n=3$ also works for any $n\ge 3$ as one expects
more freedom in the choice of the integer points $\ux$.  Indeed this is how part 1)
is proved there, by reduction to the case $n=3$ due to Schmidt.   However, it is surprising
that this exponent remains best possible, independently of $n$.   The construction that
shows its optimality and proves part 2) of the theorem generalizes the construction
elaborated in \cite{Ro2014}, and exhibits additional features (see
\cite[Lemma 4.3.10]{Ch2021}).

In Section \ref{sec:red}, we prove part 1) of Theorem \ref{intro:thm1} for the
general case $n\ge m+2$ by reducing to Thurnheer's result in the special case where
$n=m+2$.  We also provide proofs of the three corollaries assuming that the full
theorem holds.  In Section \ref{sec:simplif}, we propose a simplification of the proof
of Thurnheer from \cite{Th1990} along the lines
of \cite[Chapter 5]{Ch2021}.  It involves a reasoning which is
reminiscent of the theory of continued fractions and remarkably similar to that of
Davenport and Schmidt in \cite{DS1968} although our goal is different.  Moreover, it
allows us to treat at once the cases $m=1$ and $m\ge 2$ which were analyzed separately
by Thurnheer in \cite{Th1990}.  Preliminaries on the projective distance are gathered
in Section \ref{sec:dist}.

To prove part 2) of the theorem in the general case, we introduce a new construction in
parametric geometry of numbers which involves angular constraints.
As this requires additional discussion, we describe it in the next section.

%
%

\section{Parametric geometry of numbers with angular constraints}
\label{sec:param}

Fix an integer $n\ge 2$.
Parametric geometry of numbers, as developed in \cite{SS2013}
and \cite{Ro2015}, studies in logarithmic scale the successive minima of
the parametric families of compact symmetric convex subsets of $\bR^n$
\begin{equation}
 \label{param:eq:C}
 \cC_\uu(Q)=\{\ux\in\bR^n\,;\,\norm{\ux}\le 1
    \text{ and } |\ux\cdot\uu|\le Q^{-1}\norm{\uu} \}
 \quad (Q\ge 1)
\end{equation}
attached to non-zero points $\uu\in\bR^n$ (in the setting of \cite{Ro2015}).  Thus, its
object of study consists of the maps
\begin{equation*}
 \label{param:eq:L}
 \begin{array}{rcl}
 \uL_\uu\colon[0,\infty) &\longrightarrow &\bR^n\\
  q &\longmapsto &(L_{\uu,1}(q),\dots,L_{\uu,n}(q))
 \end{array}
\end{equation*}
where, for each $j=1,\dots,n$ and $q\ge 0$, the number $L_{\uu,j}(q)$
represents the logarithm of the $j$-th minimum of $\cC_\uu(e^q)$, namely the smallest
real number $L$ for which $e^L\cC_\uu(e^q)$ contains at least $j$ linearly
independent points of $\bZ^n$.  The main result of the theory asserts that, modulo
bounded functions, the set of these maps coincides with a simpler set of functions
called $n$-systems (or $(n,0)$-systems) whose definition is purely combinatorial
(see \cite[Section 2.5]{Ro2015}).  We can even use smaller sets of functions,
the rigid $n$-systems of a given mesh $c>0$.

To recall their definition from \cite[Section 1]{Ro2015}, let
\begin{equation*}
\label{param:eq:Delta}
 \Delta_n=\{(x_1,\dots,x_n)\in\bR^n\,;\,x_1\le\cdots\le x_n\}
\end{equation*}
denote the set of $n$-tuples of real numbers in monotone increasing order and let
$\Phi\colon\bR^n\to\Delta_n$ denote the continuous map which sends an $n$-tuple
to its permutation in $\Delta_n$.   The following definitions are adapted from
\cite[Definitions 1.1 and 1.2]{Ro2015}.

\begin{definition}
\label{param:def:canvas}
Let $c>0$, and let $s\in\{\infty,1,2,3,\dots\}$.  A \emph{canvas with mesh $c$ and cardinality
$s$} in $\bR^n$ is a triple consisting of a sequence of points $(\ua^{(i)})_{0\le i<s}$ in
$\Delta_n$ together with sequences of integers $(k_i)_{0\le i<s}$ and $(\ell_i)_{0\le i<s}$
of the same cardinality $s$ such that, for each index $i$ with $0\le i<s$,
\begin{itemize}
 \item[(C1)] the coordinates $(a^{(i)}_1,\dots,a^{(i)}_n)$ of $\ua^{(i)}$ form a strictly
  increasing sequence of positive integer multiples of $c$;\\[-5pt]
 \item[(C2)] we have $1\le k_0 < \ell_0=n$ and $1 \le k_i < \ell_i \le n$ if $i\ge 1$,\\[-5pt]
 \item[(C3)] if $i+1<s$, we further have $k_i\le \ell_{i+1}$,
  $a^{(i)}_{\ell_{i+1}}+c \le a^{(i+1)}_{\ell_{i+1}}$ and
 \begin{equation}
 \label{param:def:canvas:eq}
  (a^{(i+1)}_1,\dots,\widehat{a^{(i+1)}_{\ell_{i+1}}},\dots,a^{(i+1)}_n)
  = (a^{(i)}_1,\dots,\widehat{a^{(i)}_{k_{i}}},\dots,a^{(i)}_n)
 \end{equation}
 where the hat on a coordinate means that it is omitted.
\end{itemize}
\end{definition}

The only difference with \cite[Definition 1.1]{Ro2015} is the strict inequality $k_0<\ell_0$
in condition (C2), which is better suited to our current purposes.

Condition (C3) means that, when $i+1<s$, the point $\ua^{(i+1)}$ is obtained
from the preceding point $\ua^{(i)}$ by replacing one of its coordinates by a larger
multiple of $c$, different from its other coordinates, and then by reordering the new
$n$-tuple.

\begin{definition}
 \label{param:def:systems}
To each canvas of mesh $c>0$ as in Definition \ref{param:def:canvas},
we associate the function $\uP\colon[q_0,\infty)\to\Delta_n$ given by
\begin{equation*}
\label{param:def:systems:eq}
 \uP(q)
 =\Phi_n\big(a^{(i)}_1,\dots,\widehat{a^{(i)}_{k_i}},\dots,a^{(i)}_n,
         a^{(i)}_{k_i}+q-q_i\big)
 \quad
 (0\le i<s, \ q_i\le q<q_{i+1}),
\end{equation*}
where $q_i=a_1^{(i)}+\cdots+a_n^{(i)}$ $(0\le i<s)$ and $q_s=\infty$
if $s<\infty$.  We say that such a function is a \emph{rigid
$n$-system with mesh $c$} and that $(q_i)_{0\le i<s}$ is its
sequence of \emph{switch numbers}.
\end{definition}

Since $a^{(i)}_{k_i}+q_{i+1}-q_i=a^{(i+1)}_{\ell_{i+1}}$ when $i+1<s$,
such a map $\uP$ is continuous.  Moreover, upon writing
$\uP(q)=(P_1(q),\dots,P_n(q))$ for each $q\ge q_0$, we see that
\begin{itemize}
 \item[(S1)] $P_1,\dots,P_n$ are continuous and piecewise linear on $[q_0,\infty)$
  with slopes $0$ and $1$;\\[-5pt]
 \item[(S2)] $0\le P_1(q)\le\cdots\le P_n(q)$ and $P_1(q)+\cdots+P_n(q)=q$ for
  each $q\ge q_0$;\\[-5pt]
 \item[(S3)]  if, for some $j\in\{1,\dots,n-1\}$, the sum  $P_1+\cdots+P_j$ changes
  slope from $1$ to $0$ at a point $q>q_0$, then $P_j(q)=P_{j+1}(q)$.
\end{itemize}
Thus $\uP$ is an $(n,0)$-system in view of \cite[Section 2.5]{Ro2015}.  The switch
numbers $q_i$ with $i>0$ may also be characterized as the points $q>q_0$ at which
one of the sums $P_1+\cdots+P_j$ changes slope from $0$ to $1$.  Then,
the smallest such index $j$ is $k_i$.  From a graphical point of view,
it also follows from the
definition that, whenever $0\le i<s$, the union of the graphs of $P_1,\dots,P_n$ over
$[q_i,q_{i+1})$ consists of $n-1$ horizontal line segments and one line segment
of slope $1$.  When $i\ge 1$, the condition $k_i<\ell_i$ in (C2), means that the line
segment of slope $1$ over $[q_{i-1},q_i]$ ends at $(q_i,a^{(i)}_{\ell_i})$, above
the starting point $(q_i,a^{(i)}_{k_i})$ of the line segment of slope $1$ over
$[q_i,q_{i+1})$.

Since the convex bodies $\cC_\uu(q)$ depends only on the class of $\uu$ in projective
space, we may restrict to unit vectors $\uu$.  The main result \cite[Theorem 1.3]{Ro2015}
then reads as follows.

\begin{theorem}
 \label{param:thm:R}
Let $c>0$.
For each unit vector $\uu$ of $\bR^n$, there exists a rigid
system $\uP\colon[q_0,\infty)\to\Delta_n$ with mesh
$c$ such that $\uL_\uu-\uP$ is bounded on
$[q_0,\infty)$.  Conversely, for each rigid system
$\uP\colon[q_0,\infty)\to\Delta_n$ with mesh $c$,
there exists a unit vector $\uu$ in $\bR^n$
such that $\uL_\uu-\uP$ is bounded on $[q_0,\infty)$.
\end{theorem}

Here, we propose a new and simpler construction of the unit vector $\uu$ which
involves angular constraints and yields a simple self-contained proof of the second part of
Theorem \ref{intro:thm1}.  It assumes that the $n$-system $\uP$ satisfies
some additional properties but, even with this restriction, we will see that it suffices
for some important applications of the theory.  To state our result, we need some additional notation.

\begin{definition}
\label{param:def:traj}
Let $\uu\in\bR^n$ be a unit vector. The \emph{trajectory} of a non-zero
integer point $\ux\in\bZ^n$ (relative to the family $\cC_\uu(Q)$) is the
map $L_\uu(\ux,\cdot)\colon [0,\infty)\to\bR$ given by
\begin{equation}
\label{param:eq:Lx}
 L_\uu(\ux,q)=\max\{\log\norm{\ux},q+\log\abs{\ux\cdot\uu}\} \quad (q\ge 0).
\end{equation}
The \emph{trajectory} of a linearly independent $n$-tuple
$\uux=(\ux_1,\dots,\ux_n)$ of points of $\bZ^n$ is the
map $\uL_\uu(\uux,\cdot)\colon [0,\infty)\to\Delta_n$ given, for each $q\ge 0$, by
\begin{equation*}
\label{param:eq:Lux}
 \uL_\uu(\uux,q)
   =(L_{\uu,1}(\uux,q),\dots,L_{\uu,n}(\uux,q))
  :=\Phi_n(L_\uu(\ux_1,q),\dots,L_\uu(\ux_n,q)),
\end{equation*}
so that the coordinates $L_{\uu,j}(\uux,q)$ of $\uL_\uu(\uux,q)$ are the numbers
$L_\uu(\ux_j,q)$ written in monotone increasing order.
\end{definition}

Fix a number $q\ge 0$.  It follows from the formula \eqref{param:eq:C} for
$\cC_\uu(Q)$ that, for each non-zero $\ux\in\bZ^n$, the number
$L_\uu(\ux,q)$ given by \eqref{param:eq:Lx} is the smallest $L$ for
which $\ux\in e^L\cC_\uu(e^q)$.  Thus, we have $L_{\uu,1}(q)\le L_\uu(\ux,q)$.
More generally, we deduce that
\begin{equation}
 \label{param:eq:LL}
 L_{\uu,j}(q)\le L_{\uu,j}(\uux,q) \quad \text{for $j=1,\dots,n$,}
\end{equation}
for each linearly independent $n$-tuple $\uux$ of points of $\bZ^n$.
Since $\cC_\uu(e^q)$ is a compact subset of $\bR^n$, we can even find such
an $\uux$ that realizes the $n$ minima of $\cC_\uu(e^q)$ in the sense that
$\uL_\uu(q)=\uL_\uu(\uux,q)$.  Thus, in the componentwise ordering on $\bR^n$
defined by
\[
 (t_1,\dots,t_n) \le (t'_1,\dots,t'_n)
 \ssi t_1 \le t'_1, \ \dots\  , t_n\le t'_n,
\]
the $n$-tuple $\uL_\uu(q)$ is the minimum of $\uL_\uu(\uux,q)$ over all possible
choices of $\uux$.

\begin{definition}
 \label{param:def:bases}
Let $\uP\colon[q_0,\infty)\to\Delta_n$ be a rigid $n$-system with mesh $c$
as in Definition \ref{param:def:systems}.  Consider its associated canvas as in
Definition \ref{param:def:canvas}, and let $(q_i)_{0\le i<s}$ denote its sequence of
switch numbers.  We say that a sequence of bases $(\uux^{(i)})_{0\le i <s}$ of $\bZ^n$
written
$\uux^{(i)}=(\ux^{(i)}_1,\dots,\ux^{(i)}_n)$
is \emph{coherent with} $\uP$ if, for each $i\ge 0$ with $i+1<s$, we have
\begin{itemize}
 \item[(1)]
 $(\ux^{(i+1)}_1,\dots,\widehat{\ux^{(i+1)}_{\ell_{i+1}}},\dots,\ux^{(i+1)}_n)
   = (\ux^{(i)}_1,\dots,\widehat{\ux^{(i)}_{k_i}\,},\dots,\ux^{(i)}_n)$,\\[-5pt]
 \item[(2)]
 $\ux_{\ell_{i+1}}^{(i+1)}
  \in \ux_{k_i}^{(i)} +\langle \ux_1^{(i)},\dots,\widehat{\ux_{k_i}^{(i)}\,},\dots,\ux^{(i)}_{\ell_{i+1}} \rangle_\bZ$.
\end{itemize}
\end{definition}

Note that, if $\uux^{(i)}$ is a basis of $\bZ^n$ for some $i$ with $1\le i+1<s$, then
any choice of vectors $\ux^{(i+1)}_1,\dots,\ux^{(i+1)}_n$ satisfying conditions
(1) and (2) also form a basis of $\bZ^n$.  So, provided that $\uux^{(0)}$ is a basis
of $\bZ^n$, all subsequent $n$-tuples satisfying these conditions also form
bases of $\bZ^n$.

As for the terminology, we say that $(\uux^{(i)})_{0\le i <s}$ is coherent with $\uP$
because condition (1) in Definition \ref{param:def:bases} has the same form as
condition \eqref{param:def:canvas:eq} in (C3).  This proves very useful for induction
purposes.  Similarly, we have the following notion of a coherent system of directions
for $\uP$.

\begin{definition}
 \label{param:def:directions}
Let the notation be as in Definition \ref{param:def:bases},
and let $\uuv=(\uv_1,\dots,\uv_{n-1})$ be a linearly
independent $(n-1)$-tuple of unit vectors of $\bR^n$.
Set $\uuv^{(0)}=(\uv_1,\dots,\uv_{n-1},\uv_{k_0})$ and, for each
integer $i\ge 0$ with $i+1<s$, define recursively
$\uuv^{(i+1)}=(\uv^{(i+1)}_1,\dots,\uv^{(i+1)}_n)$ by the conditions
\begin{itemize}
 \item[(1)]
 $(\uv^{(i+1)}_1,\dots,\widehat{\uv^{(i+1)}_{\ell_{i+1}}},\dots,\uv^{(i+1)}_n)
   = (\uv^{(i)}_1,\dots,\widehat{\uv^{(i)}_{k_i}\,},\dots,\uv^{(i)}_n)$,\\[-5pt]
 \item[(2)]
 $\uv_{\ell_{i+1}}^{(i+1)} = \uv_{k_{i+1}}^{(i+1)}$.
\end{itemize}
We say that $(\uuv^{(i)})_{0\le i<s}$ is the \emph{coherent sequence of directions
for $\uP$ attached to $\uuv$}.
\end{definition}

Note that the definition of $\uuv^{(0)}$ uses the strict inequality $k_0<\ell_0=n$ in
condition (C2).  Moreover, relations (1) and (2) uniquely determine  $\uuv^{(i+1)}$
in terms of $\uuv^{(i)}$ for $i\ge 0$, because $k_{i+1}<\ell_{i+1}$ by
condition (C2).  Thus, for each integer $i\ge 0$ with
$i<s$, we have $\uv_{\ell_i}^{(i)} = \uv_{k_i}^{(i)}$ and the $(n-1)$-tuples
$(\uv^{(i)}_1,\dots,\widehat{\uv^{(i)}_{k_i}\,},\dots,\uv^{(i)}_n)$ and
$(\uv^{(i)}_1,\dots,\widehat{\uv^{(i)}_{\ell_i}\,},\dots,\uv^{(i)}_n)$
are permutations of $\uuv=(\uv_1,\dots,\uv_{n-1})$.

In the statement of our second main result below, we use the standard
Euclidean norm on the exterior powers of $\bR^n$ (recalled in Section \ref{sec:dist})
to define a constant $\theta$, and we denote by $(\ue_1,\dots,\ue_n)$ the
canonical basis of $\bR^n$.

\begin{theorem}
\label{param:thm}
Let $\uuv=(\uv_1,\dots,\uv_{n-1})$ be a linearly independent $(n-1)$-tuple of unit
vectors of $\bR^n$, and let $\delta\in\bR$ with
\begin{equation*}
 \label{param:thm:eq:c}
 0<\delta\le \frac{\theta}{4n}
\quad\text{where}\quad
 \theta=\norm{\uv_1\wedge\cdots\wedge\uv_{n-1}}.
\end{equation*}
There is a constant $\kappa>0$ depending only on $\uuv$ and $\delta$ with the following
property.  Let $\uP=(P_1,\dots,P_n)\colon [q_0,\infty)\to\Delta_n$ be a rigid
$n$-system of mesh $c>0$ with associated canvas $(\ua^{(i)})_{0\le i<s}$,
$(k_i)_{0\le i <s}$, $(\ell_i)_{0\le i <s}$ and sequence of switch numbers
$(q_i)_{0\le i<s}$ as in Definitions \ref{param:def:canvas} and \ref{param:def:systems}.
Suppose that $a^{(0)}_1=P_1(q_0)\ge \kappa$ and that at least one
of the following conditions holds
\begin{flalign}
 &\quad c\ge \log(8/\delta) \et  \ell_i=n \text{ for each integer $i$ with $0\le i<s$;} &&
 \label{param:thm:cond1}\\
 &\quad c\ge \log(2) \et \sum_{i=1}^s \exp(q_{i-1}-q_i) < \delta/4. &&
 \label{param:thm:cond2}
\end{flalign}
Then, for the coherent sequence of directions $(\uuv^{(i)})_{i\ge 0}$ attached
to $\uuv$, there is a unit vector $\uu$ of $\bR^n$ and a
coherent sequence of bases $(\uux^{(i)})_{i\ge 0}$ of $\bZ^n$ such that,
for each index $i$ with $0\le i < s$, each $j=1,\dots,n$, and each $q\in[q_i,q_{i+1})$,
we have
\begin{flalign}
 &\quad \dist(\ux_j^{(i)},\uv_j^{(i)}) \le \delta, &&
 \label{param:thm:eq1}\\
 &\quad \big|\log\norm{\ux_j^{(i)}} - a_j^{(i)}\big| \le \log(2), &&
 \label{param:thm:eq2}\\
 &\quad \big|\log|\ux_{k_i}^{(i)}\cdot\uu| - a_{k_i}^{(i)} + q_i \big| \le c_2
 \quad \text{if} \quad q_{i+1} > q_i+\log(2)+c_2, &&
 \label{param:thm:eq3}\\
 &\quad L_{\uu,j}(q) \le L_{\uu,j}(\uux^{(i)},q)
   \le P_j(q) + c_1 \le L_{\uu,j}(q)+c_2. &&
 \label{param:thm:eq4}
\end{flalign}
where $c_1=\log(32/\theta^2)$ and $c_2=nc_1+\log(n!)$.  When $\uuv=(\ue_1,\dots,\ue_{n-1})$, we have $\theta=1$ and one can take $\kappa=\log(6/\delta)$.
\end{theorem}

For $0\le i < s$ and $q\in[q_i,q_{i+1})$, the inequalities \eqref{param:thm:eq4}
imply that, in some order which depends on $q$, the basis vectors
$\ux_1^{(i)},\dots,\ux_n^{(i)}$ realize the successive minima of $\cC_\uu(e^q)$
up to the factor $\exp(c_2)$, while, by \eqref{param:thm:eq1}, each of them
stands at distance of at most $\delta$ from a vector among $\uv_1,\dots,\uv_{n-1}$.
Since \eqref{param:thm:eq4} holds independently of the choice of $i$ and $q$,
it also implies that
\begin{equation}
  \norm{\uP(q)-\uL_\uu(q)}_\infty \le c_2
 \quad \text{for each $q\ge q_0$.}
 \label{param:thm:eq5}
\end{equation}
In particular, if $P_1$ is unbounded, then $L_{\uu,1}$ is also unbounded and thus
the coordinates of $\uu$ must be linearly independent over $\bQ$.

As we will see in Section \ref{sec:cons}, the construction of the point $\uu$ is particularly
simple for a rigid $n$-system that meets condition \eqref{param:thm:cond2}.
When $s=\infty$, this
requires that the series $\sum_{i=1}^\infty \exp(q_{i-1}-q_i)$ is bounded and
thus that the difference $q_i-q_{i-1}$ tends to infinity with $i$.  Although this is restrictive,
it holds for the important class of \emph{self-similar rigid $n$-systems},
namely the rigid $n$-systems $\uP\colon[q_0,\infty)\to\Delta_n$ which, for some
$\rho>1$, satisfy
\[
 \uP(\rho q)=\rho\uP(q) \quad\text{for each $q\ge q_0$.}
\]
If such $\uP$ has mesh $c$, then, for each $h\ge 0$, the restriction of
$\uP$ to the interval $[\rho^hq_0,\infty)$ is self-similar and rigid of mesh
$\rho^hc$. Then, for $h$
large enough, this restriction fulfills conditions \eqref{param:thm:cond2}.

In \cite{Ro2017}, it is shown that the spectrum of the standard exponents of Diophantine
approximation can be computed in terms of $n$-systems only, and that the self-similar
rigid $n$-systems yield a dense subset of that spectrum.  The simplicity of the construction
of points attached to such $n$-systems could possibly help in estimating
the Hausdorff dimension of sets of points whose exponents
lie in a given region of the spectrum.

The construction of the point $\uu$ is more delicate for a rigid $n$-system that satisfies
condition \eqref{param:thm:cond1}.  This is also done in Section \ref{sec:cons}.
Although this condition restricts the shape of the $n$-system,
it is well adapted to our purpose.  In Section \ref{sec:app}, we apply this result
to specific rigid $n$-systems that satisfy condition \eqref{param:thm:cond1}
to produce the points $\uu$ that are needed in the second part of Theorem \ref{intro:thm1}.

%
%

\section{Projective distance}
\label{sec:dist}

Fix an integer $n\ge 2$.  We view $\bR^n$ as an Euclidean space for the usual
scalar product, denoted by a dot, so that the canonical basis $(\ue_1,\dots,\ue_n)$
of $\bR^n$ is orthonormal.  More generally, for each integer $k$ with $1\le k\le n$,
we endow $\bigwedge^k\bR^n$ with the unique structure of Euclidean space for
which the products $\ue_{i_1}\wedge\cdots\wedge\ue_{i_k}$ with
$1\le i_1<\cdots<i_k\le n$ form an orthonormal basis of that space, and we
denote by $\norm{\ualpha}$ the Euclidean norm of a vector $\ualpha$ in that
space.  Then the well-known Hadamard inequality tells us that
\[
 \norm{\ux_1\wedge\cdots\wedge\ux_k}\le \norm{\ux_1}\cdots\norm{\ux_k}
\]
for any choice of vectors $\ux_1,\dots,\ux_k\in\bR^n$.

We define the \emph{projective distance} between two non-zero points $\ux$, $\uy$
of $\bR^n$ by
\[
 \dist(\ux,\uy)=\frac{\norm{\ux\wedge\uy}}{\norm{\ux}\,\norm{\uy}}.
\]
It depends only on the classes of $\ux$ and $\uy$ in the projective space
over $\bR^n$ and represents the sine of the acute angle
between the lines spanned by these vectors.
In particular, it is a symmetric function of $\ux$ and $\uy$.  Moreover, it is
well-known that it satisfies the triangle inequality
\[
 \dist(\ux,\uz)\le \dist(\ux,\uy)+\dist(\uy,\uz)
\]
for any non-zero points $\ux$, $\uy$, $\uz$ of $\bR^n$
\cite[Section 8, Equation (3)]{Sc1967}.  The following property will also be useful.

\begin{lemma}
\label{dist:lemma1}
Let $\uu$, $\uu'$ be unit vectors of $\bR^n$ with $\uu\cdot\uu'\ge 0$.  Then we
have
\begin{equation*}
  \label{dist:lemma1:eq1}
 \dist(\uu,\uu') \le \norm{\uu-\uu'} \le 2\dist(\uu,\uu').
\end{equation*}
Moreover, any point $\ux\in\uu^\perp$ satisfies
\begin{equation}
\label{dist:lemma1:eq2}
 \abs{\ux\cdot\uu'}\le 2\norm{\ux}\dist(\uu,\uu').
\end{equation}
\end{lemma}

\begin{proof}
The first estimate is well-known and amounts to $\sin(\theta) \le 2\sin(\theta/2)
\le 2\sin(\theta)$ where $\theta\in[0,\pi/2]$ is the angle between $\uu$ and $\uu'$.
Then \eqref{dist:lemma1:eq2} follows since, for $\ux\in\uu^\perp$, we find
$\abs{\ux\cdot\uu'}=\abs{\ux\cdot(\uu'-\uu)}\le \norm{\ux}\norm{\uu-\uu'}$
by Hadamard's inequality.
\end{proof}

We define the projective distance from a non-zero point $\ux$ of
$\bR^n$ to a non-zero subspace $V$ of $\bR^n$ as the infimum of the
distances between $\ux$ and a non-zero point $\uy$ of $V$.  As explained in
\cite[Section 4]{Ro2015}, this infimum, denoted $\dist(\ux,V)$, is in fact
a minimum, and \cite[Lemma 4.2]{Ro2015} provides the following formula
where $\proj_{V^\perp}$ denotes the orthogonal projection on the orthogonal
complement $V^\perp$ of $V$ in $\bR^n$.

\begin{lemma}
\label{dist:lemma:xV}
For any non-zero point $\ux$ of $\bR^n$ and any non-zero subspace $V$
of $\bR^n$, we have
\[
 \dist(\ux,V) = \frac{ \norm{\proj_{V^\perp}(\ux)} }{ \norm{\ux} }
 = \frac{ \norm{\ux\wedge\uv_1\wedge\cdots\wedge\uv_m} }%
  { \norm{\ux}\,\norm{\uv_1\wedge\cdots\wedge\uv_m} },
\]
where $(\uv_1,\dots,\uv_m)$ is any basis of $V$ over $\bR$.
\end{lemma}

Finally, for any non-zero subspaces $V_1$ and $V_2$ of $\bR^n$, we define the
distance $\dist(V_1,V_2)$ from $V_1$ to $V_2$ as the supremum of the
numbers $\dist(\ux,V_2)$ with $\ux\in V_1\setminus\{0\}$.  As explained in
\cite[Section 4]{Ro2015}, this supremum is in fact a maximum.  It is not symmetric
in $V_1$ and $V_2$, for example when $V_1\varsubsetneq V_2$.  However, as
shown in \cite[Lemma 4.3]{Ro2015}, it has the following properties.

\begin{lemma}
\label{dist:lemma:xVV}
For any non-zero point $\ux$ of $\bR^n$ and any non-zero subspaces $V, V_1,V_2$
of $\bR^n$, we have
\[
 \dist(\ux,V_2) \le \dist(\ux,V_1)+\dist(V_1,V_2)
 \et
 \dist(V,V_2) \le \dist(V,V_1) + \dist(V_1,V_2).
\]
\end{lemma}

In this paper, we mainly need estimates for subspaces of $\bR^n$ of
co-dimension $1$.  For these, \cite[Lemma 4.4]{Ro2015} gives the following
alternative formula which implies that, among themselves, the distance is symmetric.

\begin{lemma}
\label{dist:lemma:uU}
For each $j=1,2$, let $U_j$ be a subspace of $\bR^n$ of codimension $1$
and let $\uu_j$ be a unit vector of $U_j^\perp$.  Then we have
$\dist(U_1,U_2) = \dist(\uu_1,\uu_2)$.
\end{lemma}

The next result provides an explicit formula for this distance in a situation
that we will encounter later.  It follows directly from \cite[Lemma 4.7]{Ro2015}
as the height of $\bR^n$ is $1$.

\begin{lemma}
\label{dist:lemma:UU}
Let $(\ux_1,\dots,\ux_n)$ be a basis of $\bZ^n$, and let $k$, $\ell$ be integers
with $1\le k<\ell\le n$.  For the subspaces
\[
 U_1=\langle\ux_1,\dots,\widehat{\ux_\ell\,},\dots,\ux_n\rangle_\bR
 \et
 U_2=\langle\ux_1,\dots,\widehat{\ux_k},\dots,\ux_n\rangle_\bR,
\]
we have
\[
 \dist(U_1,U_2)
  = \frac{\norm{\ux_1\wedge\cdots\wedge\widehat{\ux_k}
    \wedge\cdots\wedge\widehat{\ux_\ell\,}\wedge\cdots\wedge\ux_n}}%
   {\norm{\ux_1\wedge\cdots\wedge\widehat{\ux_k}\wedge\cdots\wedge\ux_n}\,
   \norm{\ux_1\wedge\cdots\wedge\widehat{\ux_\ell\,}\wedge\cdots\wedge\ux_n}}.
\]
\end{lemma}

For each $m$-tuple of non-zero vectors $\uuv=(\uv_1,\dots,\uv_m)$ of $\bR^n$
with $1\le m\le n$, we write
\begin{equation}
\label{dist:Theta}
 \Theta(\uuv)=\frac{\norm{\uv_1\wedge\cdots\wedge\uv_m}}%
    {\norm{\uv_1}\,\cdots\,\norm{\uv_m}}
 \in [0,1].
\end{equation}
This normalized volume depends only on the classes of $\uv_1,\dots,\uv_m$ in the projective
space over $\bR^n$.  We have $\Theta(\uuv)=0$ if and only if $\uuv$ is linearly dependent,
and $\Theta(\uuv)=1$ if and only if it is orthogonal.  The following result provides a
measure of continuity of this map.

\begin{lemma}
\label{dist:lemma2}
Let $\uuv=(\uv_1,\dots,\uv_m)$ be a linearly independent $m$-tuple of vectors of
$\bR^n$ for some $m\in\{1,\dots,n\}$, let $\delta\in\bR$ with
$0\le \delta<\Theta(\uuv)/(2m)$, and let $\uux=(\ux_1,\dots,\ux_m)$ be an
$m$-tuple of non-zero vectors of $\bR^n$ with $\dist(\ux_j,\uv_j)\le \delta$
for each $j=1,\dots,m$.  Then $\uux$ is linearly independent over $\bR$ and
\begin{equation*}
 \label{dist:lemma2:eq1}
 |\Theta(\uux)-\Theta(\uuv)|\le 2m\delta
\end{equation*}
Upon setting $W=\langle\ux_1,\dots,\ux_m\rangle_\bR$ and
$V=\langle\uv_1,\dots,\uv_m\rangle_\bR$, we also have
\begin{equation}
 \label{dist:lemma2:eq2}
 \dist(W,V)\le 2m\delta/\Theta(\uux).
\end{equation}
\end{lemma}

\begin{proof}
We may assume that, for each $j=1,\dots,m$, we have
$\norm{\uv_j}=\norm{\ux_j}=1$ and $\ux_j\cdot\uv_j\ge 0$ so that Lemma
\ref{dist:lemma1} gives $\norm{\ux_j-\uv_j} \le 2\delta$.  Set
\[
 \ualpha_j=\uv_1\wedge\cdots\wedge\uv_j\wedge\ux_{j+1}\wedge\cdots\wedge\ux_m
 \quad
 \text{for $j=0,\dots,m$.}
\]
For each $j=0,\dots,m-1$, we find that
\[
 \norm{\ualpha_{j}-\ualpha_{j+1}}
  = \norm{\uv_1\wedge \cdots\wedge  \uv_j \wedge (\ux_{j+1}-\uv_{j+1})
      \wedge\ux_{j+2}\wedge\cdots\wedge\ux_m}
 \le 2\delta,
\]
thus $|\Theta(\uux)-\Theta(\uuv)|=|\,\norm{\ualpha_0}-\norm{\ualpha_m}\,|
\le \norm{\ualpha_0-\ualpha_m}\le 2m\delta$.  Since $2m\delta<\Theta(\uuv)$, this
implies that $\Theta(\uux)>0$, thus $\uux$ is linearly independent
over $\bR$.

To prove the last estimate, choose any non-zero $m$-tuple $(a_1,\dots,a_m)\in\bR^m$,
and form the points $\ux=a_1\ux_1+\cdots+a_m\ux_m\in W$ and
$\uv=a_1\uv_1+\cdots+a_m\uv_m\in V$.  For each $j=1,\dots,m$, we find,
using Hadamard's inequality, that
\[
 \norm{\ux}
 \ge \norm{\ux\wedge\ux_1\wedge\cdots\wedge\widehat{\ux_j\,}\wedge\cdots\wedge\ux_m}
  = |a_j| \Theta(\uux).
\]
This implies that
\[
 \norm{\ux-\uv}\le 2\delta(|a_1|+\dots+|a_m|)\le 2m\delta\norm{\ux}/\Theta(\uux),
\]
and so
\[
 \dist(\ux,V)
  = \frac{\norm{\proj_{V^\perp}(\ux)}}{\norm{\ux}}
  \le \frac{\norm{\ux-\uv}}{\norm{\ux}}
  \le 2m\delta/\Theta(\uux).
\]
As $\ux$ can be any non-zero point of $W$, this gives \eqref{dist:lemma2:eq2}.
\end{proof}

The last lemma below will be needed in Section \ref{sec:cons} to prove
part 2) of Theorem \ref{intro:thm1}.

\begin{lemma}
\label{dist:lemma3}
Let $r$ be an integer with $1\le r\le n$, let  $(\uv_1,\dots,\uv_n)$ be an
orthonormal basis of $\bR^n$ and let
\[
 V=\langle\uv_1+\cdots+\uv_r,\uv_{r+1},\dots,\uv_n\rangle_\bR.
\]
Suppose that non-zero vectors $\ux,\ux_1,\dots,\ux_r\in\bR^n$ satisfy
\[
 \dist(\ux,V)\le \delta
 \et
 \max_{1\le j\le r}\dist(\ux_j,\uv_j) \le\delta
\]
for some $\delta\in\bR$ with $0\le \delta\le 1/(24r)$, and that
\[
 \ux=a_1\ux_1+\cdots+a_r\ux_r
\]
for some $a_1,\dots,a_r\in\bR$.  Then, for each $j=1,\dots,r$, we have
\begin{equation*}
 \label{dist:lemma3:eq1}
 \frac{\norm{\ux}}{2\sqrt{r}}
 \le \norm{a_j\ux_j}
 \le \frac{2\norm{\ux}}{\sqrt{r}}.
\end{equation*}
\end{lemma}

\begin{proof}
Upon dividing $\ux,\ux_1,\dots,\ux_r$ by their norms, we may assume that
these vectors have norm $1$.  Then, upon multiplying each of them appropriately
by $\pm 1$, we may further assume, by Lemma \ref{dist:lemma1}, that
\[
 \norm{\ux-\uv}\le 2\delta
 \et
 \max_{1\le j\le r}\norm{\ux_j-\uv_j} \le 2\delta
\]
for a unit vector $\uv$ of $V$ which is closest to $\ux$, of the form
\[
 \uv = b(\uv_1+\cdots+\uv_r)+\uw
\]
with $b\ge 0$ and $\uw\in W$, where $W=\langle \uv_{r+1},\dots,\uv_n\rangle_\bR$.
Upon setting
\[
 \uy=a_1\uv_1+\cdots+a_r\uv_r\in W^\perp
 \et
 A = 1+|a_1|+\cdots+|a_r|,
\]
we find
\begin{align*}
 \norm{\uv-\uy}
  &=\norm{(\uv-\ux)+a_1(\ux_1-\uv_1)+\cdots+a_r(\ux_r-\uv_r)}
 \le 2A\delta, \\
 \abs{a_j-b} &=\abs{(\uy-\uv)\cdot\uv_j} \le \norm{\uv-\uy}
  \quad\text{for $j=1,\dots,r$,}\\
 \abs{b\sqrt{r}-1}
   &=\abs{ \norm{b(\uv_1+\cdots+\uv_r)} - \norm{\uv} }
   \le \norm{\uw}.
\end{align*}
Since $\uw=\proj_W(\uv)=\proj_W(\uv-\uy)$, we also have
$\norm{\uw}\le \norm{\uv-\uy}$ and so the three preceding estimates
give
\begin{equation}
 \label{dist:lemma3:eq5}
 \max_{1\le j\le r}\abs{a_j-1/\sqrt{r}}
  \le (1+1/\sqrt{r}) \norm{\uv-\uy}\le 4A\delta.
\end{equation}
In turn, by definition of $A$, this implies that
\[
 A\le 1+r(1/\sqrt{r}+4A\delta) \le 2\sqrt{r}+A/6
\]
since $4r\delta\le 1/6$.  So, we obtain the upper bound $A\le 3\sqrt{r}$
which substituted in \eqref{dist:lemma3:eq5} gives
\[
  \max_{1\le j\le r}\abs{a_j-1/\sqrt{r}}
  \le (12r\delta)/\sqrt{r}\le 1/(2\sqrt{r}).
\]
and so $1/(2\sqrt{r})\le |a_j| \le 2/\sqrt{r}$ for $j=1,\dots,r$.
\end{proof}

%
%

\section{Proof of the first part of Theorem \ref{intro:thm1}
  and deduction of the corollaries}
\label{sec:red}

In this section, we prove the first part of our first main theorem \ref{intro:thm1}
and, assuming that its second part holds as well, we deduce the corollaries stated in
the introduction.  To this end, we work with a fixed integer $n\ge 2$, and we identify
$\GL_n(\bZ)$ with the subgroup of $\GL_n(\bR)$ consisting of the linear
operators $T$ on $\bR^n$ with $T(\bZ^n)=\bZ^n$.  The first lemma below
is our main tool.

\begin{lemma}
\label{red:lemma1}
Let $\uv_1,\dots,\uv_m$ be linearly independent vectors of $\bR^n$ for some integer
$m$ with $0\le m<n$, and let $\delta>0$.  Then, there exists $T\in\GL_n(\bZ)$ such that
$\dist(T(\ue_j),\uv_j)\le \delta$ for each index $j$ with $1\le j\le m$.
\end{lemma}

Note that the restriction $m<n$ is needed in general.  For $m=1$, the lemma is a
consequence of a much stronger result of Erd\"os \cite{Er1958}.

\begin{proof}
We are going to prove, by induction on $m$, that there is a basis
$(\ux_1,\dots,\ux_n)$ of $\bZ^n$ such that $\dist(\ux_j,\uv_j)\le \delta$ for
each $j$ with $1\le j\le m$.  The result then follows by taking for $T$ the element
of $\GL_n(\bZ)$ which sends $\ue_j$ to $\ux_j$ for each $j=1,\dots,n$.

For $m=0$, we can take any basis $(\ux_1,\dots,\ux_n)$ of $\bZ^n$.
Suppose now that $1\le m<n$.  We may assume, by induction, that
there is a basis $(\uy_1,\dots,\uy_n)$ of $\bZ^n$
such that $\dist(\uy_j,\uv_j)\le \delta$ for each $j$ with $1\le j\le m-1$.  Write
\[
 \uv_m=a_1\uy_1+\cdots+a_n\uy_n
\]
with $a_1,\dots,a_n\in\bR$.  By permuting $\uy_m,\dots,\uy_n$ if necessary,
and by multiplying each of them by $\pm 1$ appropriately, we may assume that
$0\le a_m\le \cdots\le a_n$.  For a choice of $t>0$ to be fixed later,
we form the point
\[
 \ux_m=b_1\uy_1+\cdots+b_n\uy_n \quad\text{where}\quad
 b_j=\begin{cases}
           \lceil ta_j\rceil &\text{if $j\neq m$,}\\
           \lceil ta_j\rceil - 1 &\text{if $j=m$.}
       \end{cases}
\]
By construction, we have $\norm{\ux_m-t\uv_m}\le Y$ where
$Y=\norm{\uy_1}+\cdots+\norm{\uy_n}$ is independent of $t$.  Thus, if $t$ is
large enough, say $t\ge t_0$, we obtain
\[
 \dist(\ux_m,\uv_m)
  =\frac{\norm{(\ux_m-t\uv_m)\wedge \uv_m}}{\norm{\ux_m}\,\norm{\uv_m}}
 \le \frac{\norm{\ux_m-t\uv_m}}{\norm{\ux_m}}
 \le \frac{Y}{t\norm{\uv_m}-Y}
 \le \delta.
\]
If $a_n\neq0$, we take $t=p/a_n$ for a prime number $p$ with $p\ge a_nt_0$, so that
the above inequality holds.  Then, we have $b_n=p$ and $-1\le b_m<p$.  So, the
integers $b_m,\dots,b_n$ are relatively prime as a set.  If $a_n=0$, we take $t=t_0$
and then $(b_m,\dots,b_n)=(-1,0,\dots,0)$.  So, in both cases, the point
$\ux=b_m\uy_m+\cdots+b_n\uy_n$ can be completed to a basis
$(\ux,\ux_{m+1},\dots,\ux_n)$ of $\langle \uy_m,\dots,\uy_n\rangle_\bZ$. Then,
$(\uy_1,\dots,\uy_{m-1},\ux_m,\dots,\ux_n)$ is a basis of $\bZ^n$ with the required
properties.
\end{proof}

\begin{lemma}
\label{red:lemma2}
Let $T\in\GL_n(\bR)$.  There exists a constant $\kappa\ge 1$ with the following
properties:
\begin{itemize}[labelindent=10pt, leftmargin=*]
\item[$1)$] $\kappa^{-1}\norm{\ux}\le \norm{T(\ux)}\le \kappa\norm{\ux}$
 for each $\ux\in\bR^n$;
\medskip
\item[$2)$] $\kappa^{-1}\dist(\ux,\uy)\le \dist(T(\ux),T(\uy))\le \kappa\dist(\ux,\uy)$
 for any non-zero $\ux,\uy\in\bR^n$;
\medskip
\item[$3)$] $\kappa^{-1}\dist(\ux,V)\le \dist(T(\ux),T(V))\le \kappa\dist(\ux,V)$
 for any non-zero $\ux\in\bR^n$ and any non-zero subspace $V$ of $\bR^n$;
\end{itemize}
\end{lemma}

\begin{proof}
For each $m=1,\dots,n$, the linear operator $\bigwedge^mT$ on $\bigwedge^m\bR^n$
is invertible and so there is a constant $c_m\ge 1$ such that
\[
 c_m^{-1}\norm{\ux_1\wedge\cdots\wedge\ux_m}
 \le \norm{T(\ux_1)\wedge\cdots\wedge T(\ux_m)}
 \le c_m \norm{\ux_1\wedge\cdots\wedge\ux_m}
\]
for any $\ux_1,\dots,\ux_m\in\bR^n$.  Thus, for a non-zero point $\ux$ of $\bR^n$
and a non-zero proper subspace $V$ of $\bR^n$ with basis $(\uv_1,\dots,\uv_m)$,
we find, using Lemma \ref{dist:lemma:xV}, that
\begin{align*}
 \dist(T(\ux),T(V))
 &= \frac{\norm{T(\ux)\wedge T(\uv_1)\wedge\cdots\wedge T(\uv_m)}}%
   {\norm{T(\ux)}\,\norm{T(\uv_1)\wedge\cdots\wedge T(\uv_m)}}\\
 &\le \frac{c_{m+1}\norm{\ux\wedge\uv_1\wedge\cdots\wedge\uv_m}}%
   {c_1^{-1}c_m^{-1}\norm{\ux}\,\norm{\uv_1\wedge\cdots\wedge\uv_m}}
 = c_1c_mc_{m+1}\dist(\ux,V),
\end{align*}
 and similarly $\dist(T(\ux),T(V))\ge (c_1c_mc_{m+1})^{-1}\dist(\ux,V)$.
If $V=\bR^n$, then $T(V)=\bR^n$ and so $\dist(T(\ux),T(V))=\dist(\ux,V)=0$.
Thus, properties 1) and 3) hold for an appropriate choice of $\kappa$.  Then property 2) follows
by taking $V=\langle\uy\rangle_\bR$.
\end{proof}

In the arguments below, we denote by $\trT$ the transpose of a linear operator $T$
on $\bR^n$, namely the linear operator on $\bR^n$ characterized by
$T(\ux)\cdot\uy=\ux\cdot\trT(\uy)$ for any $\ux,\uy\in\bR^n$.

\subsection*{Proof of Theorem \ref{intro:thm1}, part 1)}  When $n=m+2$ and
$V=\langle\ue_1,\dots,\ue_{m+1}\rangle_\bR \subseteq \bR^{m+2}$, that statement
follows from \cite[Theorem 1 (b)]{Th1990}.  There, Thurnheer shows that, for
any $\delta,\epsilon>0$ and any point $\uu\in\bR^{m+2}$ with $\bQ$-linearly
independent coordinates, there exists a non-zero point $\ux\in\bZ^{m+2}$ such that
\[
 |\ux\cdot\ue_{m+2}|\le\delta\norm{\ux} \et |\ux\cdot\uu|\le\epsilon\norm{\ux}^{-\rho}
\]
where $\rho$ is given by \eqref{intro:thm1:eq:rho}.  However, since
$V=\ue_{m+2}^\perp$, the first inequality may be rewritten as $\dist(\ux,V)\le \delta$.
We give a simplified, self-contained proof of this result in Section \ref{sec:simplif}.

In general, suppose that $n\ge m+2\ge 3$, that $V$ is an arbitrary subspace of $\bR^n$
of dimension $m+1$ and that $\uu\in\bR^n$ has linearly independent coordinates over
$\bQ$.  Fix $\delta,\epsilon>0$ with $\delta<1$ and choose an orthonormal basis
$\uuv=(\uv_1,\dots,\uv_{m+1})$ of $V$.  By Lemma \ref{red:lemma1}, there exists
 $T\in\GL_n(\bZ)$ such that
\[
 \dist(T(\ue_j),\uv_j) \le \frac{\delta}{8(m+1)} \quad\text{for $j=1,\dots,m+1$.}
\]
Set $\uux=(T(\ue_1),\dots,T(\ue_{m+1}))$.   By Lemma \ref{dist:lemma2}, we have
$\Theta(\uux)\ge \Theta(\uuv)-\delta/4 \ge 1/2$ and so
\begin{equation}
 \label{red:proof:thm1:eq1}
 \dist(T(\tV),V) \le \delta/2
 \quad\text{where}\quad
 \tV=\langle\ue_1,\dots,\ue_{m+1}\rangle_\bR\subseteq \bR^n.
\end{equation}
For our choice of $T$, choose a constant $\kappa\ge 1$ as in Lemma \ref{red:lemma2},
and define
\[
 \tuu=\trT(\uu)\in\bR^n.
\]
Since $T\in\GL_n(\bZ)$, the $n$ coordinates of
$\tuu$ are linearly independent over $\bQ$.  A fortiori, its first $m+2$ coordinates are
linearly independent over $\bQ$.  Thus, under the natural identification of $\bR^{m+2}$
with $\langle\ue_1,\dots,\ue_{m+2}\rangle_\bR \subseteq \bR^n$, the result of
Thurnheer provides
a non-zero point $\tux\in\langle\ue_1,\dots,\ue_{m+2}\rangle_\bZ$ such that
\[
 \dist(\tux,\tV) \le \delta/(2\kappa)
 \et
 |\tux\cdot\tuu| \le \epsilon \kappa^{-\rho}\norm{\tux}^{-\rho}.
\]
Then $\ux=T(\tux)$ is a non-zero point of $\bZ^n$ which, by definition of $\tuu$,
satisfies $\ux\cdot\uu=\tux\cdot\tuu$.  Therefore, by Lemma \ref{red:lemma2}, the
preceding inequalities yield
\[
 \dist(\ux,T(\tV)) \le \kappa\dist(\tux,\tV) \le \delta/2
 \et
 |\ux\cdot\uu|  \le \epsilon \kappa^{-\rho}\norm{\tux}^{-\rho}
  \le \epsilon \norm{\ux}^{-\rho}.
\]
Then, using \eqref{red:proof:thm1:eq1}, Lemma \ref{dist:lemma:xVV} provides
\[
 \dist(\ux,V) \le \dist(\ux,T(\tV))+\dist(T(\tV),V) \le \delta
\]
as needed.

\begin{lemma}
\label{red:lemma3}
Let $V_1$, $V_2$ be non-zero subspaces of $\bR^n$.  There is a constant $\kappa$
depending only on $V_1$ and $V_2$ such that any non-zero point $\ux$ of $\bZ^n$
satisfies
\begin{equation}
\label{red:lemma3:eq}
 \dist(\ux,V_1\cap V_2)\le \kappa\max\{\dist(\ux,V_1),\dist(\ux,V_2)\},
\end{equation}
with the convention that the left hand side is $1$ if $V_1\cap V_2=\{0\}$.
\end{lemma}

\begin{proof}
There are integers $r,s,t\in\{1,\dots,n\}$ with $r\le t$ and an invertible linear
operator $T\in\GL_n(\bR)$ such that $T(E_1)=V_1$ and $T(E_2)=V_2$, where
\[
 E_1=\langle\ue_1,\dots,\ue_s\rangle_\bR
 \et
 E_2=\langle\ue_r,\dots,\ue_t\rangle_\bR.
\]
So, by Lemma \ref{red:lemma2}, it suffices to prove the result for $V_1=E_1$
and $V_2=E_2$.  If $E_1\cap E_2\neq\{0\}$, we have $r\le s$.  Then, for each
non-zero $\ux=(x_1,\dots,x_n)\in\bR^n$, we find that
\begin{align*}
 \proj_{(E_1\cap E_2)^\perp}(\ux)&=(x_1,\dots,x_{r-1},x_{s+1},\dots,x_n),\\
 \proj_{E_1^\perp}(\ux)&=(x_{s+1},\dots,x_n),\\
 \proj_{E_2^\perp}(\ux)&=(x_1,\dots,x_{r-1},x_{t+1},\dots,x_n),
\end{align*}
thus  $\norm{ \proj_{(E_1\cap E_2)^\perp}(\ux)} \le \norm{\proj_{E_1^\perp}(\ux)}+\norm{\proj_{E_2^\perp}(\ux)}$, and so
\[
 \dist(\ux,E_1\cap E_2)\le \dist(\ux,E_1)+\dist(\ux,E_2).
\]
If $E_1\cap E_2=\{0\}$, then $s<r$ and we obtain
$\norm{\ux} \le \norm{\proj_{E_1^\perp}(\ux)}+\norm{\proj_{E_2^\perp}(\ux)}$,
which yields
\[
 1\le \dist(\ux,E_1)+\dist(\ux,E_2).
\]
So, in this situation, \eqref{red:lemma3:eq} holds with $\kappa=2$.
\end{proof}


\subsection*{Proof of Corollary \ref{intro:cor0}}
Let $n$, $m$, $\rho$ and $V$ be as in the statement of Theorem \ref{intro:thm1}.
For part 1), fix a point $\uu$ of $\bR^n$ with linearly independent coordinates
over $\bQ$, and real numbers $\delta,\epsilon>0$.  Setting $W=V\cap\uu^\perp$,
Lemma \ref{red:lemma3} provides a constant $\kappa>0$ such that
\[
 \dist(\ux,W)\le \kappa\max\{\dist(\ux,V),\dist(\ux,\uu^\perp)\},
\]
for any non-zero $\ux\in\bZ^n$.  Define $\delta'=\delta/\kappa$
and $\epsilon'=\min\{\epsilon,\delta'\norm{\uu}\}$.  Part 1) of Theorem
\ref{intro:thm1} provides a non-zero point $\ux\in\bZ^n$ with
\[
 \dist(\ux,V)\le \delta' \et |\ux\cdot\uu|\le\epsilon'\norm{\ux}^{-\rho}.
\]
Since $\dist(\ux,\uu^\perp)=|\ux\cdot\uu|/(\norm{\ux}\,\norm{\uu})
\le \epsilon'/\norm{\uu} \le \delta'$, this point $\ux$ satisfies
$\dist(\ux,W)\le\delta$ and $|\ux\cdot\uu|\le\epsilon\norm{\ux}^{-\rho}$,
as needed.  For part 2), the assertion of the corollary is clear because
$\dist(\ux,V)\le \dist(\ux,W)$ for any non-zero point $\ux\in\bR^n$ and any
non-zero subspace $W$ of $V$.

A similar argument allows to recover Theorem \ref{intro:thm1} from
its consequence provided by Corollary \ref{intro:cor0}.  This consequence
is thus an equivalent form of the theorem.


\subsection*{Proof of Corollary \ref{intro:cor1}}
Here we have $n\ge 3$.  To prove the first assertion of the corollary, fix a point
$\uu\in\bR^n$ with linearly independent coordinates over $\bQ$ and fix
$\epsilon>0$.  The subspace $V$ of $\bR^n$ given by
\eqref{intro:choixV:eq1} is defined over $\bQ$ of dimension $2$, thus
$W=V\cap\uu^\perp$ has dimension $1$ and so it is spanned by a unit
vector $\uw=(a,b,\dots,b)$ for some non-zero $a,b\in \bR$ with $b>0$.
Then Corollary \ref{intro:cor0} provides a non-zero point
$\ux=(x_1,\dots,x_n)\in\bZ^n$ with
\[
 \dist(\ux,\uw)\le b/4 \et |\ux\cdot\uu|\le \epsilon\norm{\ux}^{-\rho},
\]
where $\rho=\rho_1=\gamma$.
Replacing $\ux$ by $-\ux$ if necessary, we may further assume that
$\ux\cdot\uw\ge 0$.  Then, by Lemma \ref{dist:lemma1}, the difference
$\norm{\ux}^{-1}\ux-\uw$ has norm at most $b/2$ and so we obtain
$x_j\ge b\norm{\ux}/2>0$ for $j=2,\dots,n$.  Thus, the point $\ux$ satisfies
\eqref{intro:cor1:eq1}.

To prove the second assertion, fix an unbounded monotonically increasing function
$\psi$ from $[1,\infty)$ to $(0,\infty)$.  Assuming that part 2) of
Theorem \ref{intro:thm1} holds, there is a point $\tuu$ of $\bR^n$ with linearly
independent coordinates over $\bQ$ and a number $\tdelta>0$ for which at most
finitely many points $\tux\in\bZ^n$ satisfy
\begin{equation}
\label{red:proof:cor1:eq}
\dist(\tux,V)\le \tdelta
 \et
 \abs{\tux\cdot\tuu}\le \psi(\norm{\tux}^{1/2})^{-1/2}\norm{\tux}^{-\rho},
\end{equation}
with $V$ given by \eqref{intro:choixV:eq1}, and  $\rho=\rho_1=\gamma$.
Choose a positive  integer
$k$ such that $3/k\le \tdelta$ and form the map $T\in\GL_n(\bZ)$ defined,
for any $\ux=(x_1,\dots,x_n)\in\bR^n$, by
\[
 T(\ux)=(x_1,x_3+(k+1)\xbar,x_3+k\xbar,\dots,x_n+k\xbar)
 \quad\text{where}\quad \xbar=x_2+\dots+x_n.
\]
We claim that the point $\uu=\trT(\tuu)$ has the required properties.  Since
its coordinates are linearly independent over $\bQ$, we need to show that
there are at most finitely many points $\ux=(x_1,\dots,x_n)\in\bZ^n$ with
property \eqref{intro:cor1:eq2}, namely
\begin{equation*}
 x_2,\dots,x_n>0
 \et
 \abs{\ux\cdot\uu}\le \psi(\norm{\ux})^{-1}\norm{\ux}^{-\rho}
\end{equation*}
(as $\rho=\gamma$).  For such a point $\ux$, the integer $\xbar$ is positive
and, since $\uv=(x_1,k\xbar,\dots,k\xbar)$ belongs to $V$, we find
\[
 \dist(T(\ux),V)
  \le \frac{\norm{T(\ux)-\uv}}{\norm{T(\ux)}}
  = \frac{\norm{(0,x_3+\xbar,x_3,\dots,x_n)}}{\norm{T(\ux)}}
  \le \frac{3\xbar}{k\xbar} \le \tdelta.
\]
We also have $\norm{\ux}\le \norm{T(\ux)}\le\kappa\norm{\ux}$ where
$\kappa=n(k+1)$, thus
\begin{equation}
\label{red:proof:cor1:eq2}
 |T(\ux)\cdot\tuu|
   = |\ux\cdot\uu|
  \le \psi(\norm{\ux})^{-1}\norm{\ux}^{-\rho}
  \le \psi(\kappa^{-1}\norm{T(\ux)})^{-1}\kappa^\rho\norm{T(\ux)}^{-\gamma}.
\end{equation}
Thus, if the norm of $\ux$ is large enough, the point $\tux=T(\ux)$ satisfies
\eqref{red:proof:cor1:eq}.  So, $\tux$ and $\ux$ lie in finite sets.


\subsection*{Proof of Corollary \ref{intro:cor2}}
The argument is similar to the proof of the preceding corollary.  Let $V$ be the
subspace of $\bR^n$ defined by \eqref{intro:choixV:eq2} for the given integers
$m$ and $n$.  For the first assertion, fix $\delta,\epsilon>0$ and a point
$\uu=(u_1,\dots,u_n)\in\bR^n$ with linearly independent coordinates over
$\bQ$.  Since $u_1\neq 0$, there is a constant $\kappa>0$ depending on $\uu$
such that
\begin{equation}
 \label{red:proof:cor2:eq1}
 \norm{\ux} \le \kappa\max\{|\ux\cdot\uu|,|x_2|,\dots,|x_n|\}
\end{equation}
for each $\ux=(x_1,\dots,x_n)\in\bR^n$.  Set $\epsilon'=\min\{\epsilon,1/(2\kappa)\}$.
Part 1) of Theorem \ref{intro:thm1} provides a non-zero point $\ux$ of $\bZ^n$ with
\[
 \dist(\ux,V)\le \delta/\kappa
 \et
 |\ux\cdot\uu|\le \epsilon'\norm{\ux}^{-\rho}.
\]
For this point, we have $\kappa |\ux\cdot\uu| \le 1/2 <\norm{\ux}$. Thus, upon
writing $\ux=(x_1,\dots,x_n)$, estimate \eqref{red:proof:cor2:eq1} yields
\[
 \norm{\ux} \le \kappa\max\{|x_2|,\dots,|x_n|\}
\]
On the other hand, we have
\begin{equation}
 \label{red:proof:cor2:eq}
 \dist(\ux,V) = \frac{\norm{(x_{m+2},\dots,x_n)}}{\norm{\ux}}
  \ge \frac{\max\{|x_{m+2}|,\dots,|x_n|\}}{\norm{\ux}}.
\end{equation}
So, the point $\ux$ has the required property \eqref{intro:cor2:eq1}.

For the second assertion, fix an unbounded monotonically increasing function
$\psi$ from $[1,\infty)$ to $(0,\infty)$.
Part 2) of Theorem \ref{intro:thm1} provides a point $\tuu$ of $\bR^n$ with
linearly independent coordinates over $\bQ$ and a number $\tdelta>0$ for which
at most finitely many points $\tux\in\bZ^n$ satisfy \eqref{red:proof:cor1:eq}
for our current choice of $V$.
Set $k=\lceil \sqrt{n}/\tdelta\,\rceil$ and form the map $T\in\GL_n(\bR)$
defined, for any $\ux=(x_1,\dots,x_n)\in\bR^n$, by
\[
 T(\ux)=(kx_1,kx_2,\dots,kx_{m+1},x_{m+2},\dots,x_n).
\]
We claim that the point $\uu=\trT(\tuu)$ has the required properties.  To prove
this, suppose that a non-zero point $\ux=(x_1,\dots,x_n)$ of $\bZ^n$ satisfies
\eqref{intro:cor2:eq2}.  As in the proof of the second assertion of Corollary
\ref{intro:cor1}, we simply need to show that $\tux=T(\ux)\in\bZ^n$ satisfies
\eqref{red:proof:cor1:eq} when $\ux\in\bZ^n$ has sufficiently large norm.
The first condition in \eqref{intro:cor2:eq2} yields
\[
 \norm{(x_{m+2},\dots,x_n)}
  \le \sqrt{n}\max\{|x_1|,\dots,|x_n|\}
  =  \sqrt{n}\max\{|x_1|,\dots,|x_{m+1}|\},
\]
thus
\[
 \dist(T(\ux),V)=\frac{\norm{(x_{m+2},\dots,x_n)}}{\norm{T(\ux)}}
 \le \frac{\sqrt{n}\max\{|x_1|,\dots,|x_{m+1}|\}}{k\norm{(x_1,\dots,x_{m+1})}}
 \le \frac{\sqrt{n}}{k}
 \le \tdelta.
\]
Since $\norm{\ux}\le \norm{T(\ux)}\le k\norm{\ux}$, we also find that
\eqref{red:proof:cor1:eq2} holds with $\kappa=k$.  Thus $\tux=T(\ux)$
satisfies \eqref{red:proof:cor1:eq} if $\norm{\ux}$ is large enough.


\subsection*{Proof of Corollary \ref{intro:cor3}}
The first part of Corollary \ref{intro:cor2} shows that $\rho(\uu)\ge \rho_m$
for each point $\uu\in\bR^n$ with linearly independent coordinates over $\bQ$:
it suffices to choose $\delta=1/2$ and let $\epsilon$ tend to $0$.  Its second part, applied
with $\psi(t)=\log(t+1)$, provides a point $\uu$ with $\bQ$-linearly independent
coordinates and $\rho(\uu)\le \rho_m$, thus $\rho(\uu)=\rho_m$.  Since
$\rho(t\uu)=\rho(\uu)$ for any $t>0$, we may further choose $\uu$ of norm $1$.

By the above, the set $\cS$ of unit vectors $\uu\in\bR^n$ with $\bQ$-linearly independent
coordinates for which $\rho(\uu)=\rho_m$ is not empty.  To show that it is uncountable,
choose an arbitrary sequence $(\uu_i)_{i\ge 1}$ in $\cS$.  For each index $i$,
let $\psi_i\colon[1,\infty)\to(0,\infty)$ be the function given by
\[
 \psi_i(t)=\max\{ |\ux\cdot\uu_i|^{-1}\norm{\ux}^{-\rho_m}
     \,;\, \ux\in\cX \ \text{and}\ \norm{\ux}\le t\}
\]
where $\cX$ stands for the set of all non-zero $\ux\in\bZ^n$ satisfying the main condition
\eqref{intro:cor3:eq} of Corollary \ref{intro:cor3}.  Each $\psi_i$ is monotonically increasing
and, by Corollary \ref{intro:cor2}, it is unbounded.  Based on this, we construct
recursively a sequence of real numbers $(t_i)_{i\ge 0}$ by setting first $t_0=1$
and, for $i\ge 1$, by choosing $t_i\ge 2t_{i-1}$ such that
\[
 \min\{\psi_1(t_i),\dots,\psi_i(t_i)\}\ge (i+1)^2.
\]
Since this sequence is strictly increasing and unbounded, we obtain a monotonically
increasing function $\psi\colon[1,\infty)\to[1,\infty)$ by defining $\psi(t)=i$ when
$t\in[t_{i-1},t_i)$ for an integer $i\ge 1$.  This implies that
\[
 \psi(t)\le \sqrt{\psi_i(t)} \quad\text{for each $i\ge 1$ and each $t\ge t_i$}
\]
because, for $t\in[t_{j-1},t_j)$ with $j>i$, we have $\psi_i(t)\ge \psi_i(t_{j-1})\ge
j^2=\psi(t)^2$.  Let $\uu$ be as in the second part of Corollary \ref{intro:cor2}
for this choice of $\psi$ and let $\tuu=\kappa^{-1}\uu$ where $\kappa=\norm{\uu}$,
so that $\tuu$ is a unit vector in $\cS$.  We claim that $\tuu\neq\uu_i$ for each
$i\ge 1$.  Indeed, for a given $i\ge 1$, there are elements $\ux$ of $\cX$ of
arbitrarily large norm with
\[
 |\ux\cdot\uu_i| = \psi_i(\norm{\ux})^{-1} \norm{\ux}^{-\rho_m}.
\]
Choosing $\ux$ outside of the exceptional set for $\psi$, with $\norm{\ux}\ge t_i$
and $\psi_i(\norm{\ux})\ge\kappa^2$, this yields
\[
 |\ux\cdot\uu_i| \le \kappa^{-1}\psi_i(\norm{\ux})^{-1/2} \norm{\ux}^{-\rho_m}
  \le \kappa^{-1}\psi(\norm{\ux})^{-1} \norm{\ux}^{-\rho_m}
  < \kappa^{-1}|\ux\cdot\uu|
  = |\ux\cdot\tuu|,
\]
thus $\tuu\neq \uu_i$.  This shows that $(\uu_i)_{i\ge 1}$ does not exhaust
$\cS$.  As this is an arbitrary sequence in $\cS$, this set is uncountable.

%
%

\section{Proof of Theorem \ref{param:thm}}
\label{sec:cons}

Throughout this section, we fix a canvas ($(\ua^{(i)})_{0\le i<s}$,
$(k_i)_{0\le i<s}$, $(\ell_i)_{0\le i<s}$) of mesh $c>0$ and cardinality
$s\in\{\infty,1,2,\dots\}$, as in Definition \ref{param:def:canvas}.
We form its associated rigid
$n$-system $\uP=(P_1,\dots,P_n)\colon [q_0,\infty) \to\bR^n$, and we
denote by $(q_i)_{0\le i<s}$ its sequence of switch numbers, as in
Definition \ref{param:def:systems}.   We first establish a proposition which
exhibits the driving principle behind the original constructions
from \cite[Section 5]{Ro2015}.  Then, we use it to prove Theorem \ref{param:thm}
for a choice of directions $\uuv=(\uv_1,\dots,\uv_{n-1})$.
To simplify the writing, we set
\[
 A^{(i)}_j=\exp(a^{(i)}_j)=\exp(P_j(q_i))
 \quad (0\le i<s,\ 1\le j\le n).
\]
By condition (C1) in Definition \ref{param:def:canvas}, these numbers satisfy
\begin{equation}
\label{cons:eq:A}
 A_1^{(i)}\ge \exp(c) \et A_j^{(i)}\ge\exp(c)A_{j-1}^{(i)} \text{ if $j>1$.}
\end{equation}

The following proposition is inspired from \cite[Section 5]{Ro2015}, and
involves a parameter $t$ to allow a recursive construction of bases.

\begin{proposition}
\label{cons:prop1}
Let $t \in \{\infty,1,2,\dots\}$ with $t\le s$ and let $\theta\in(0,1]$.
Suppose that $c\ge\log(2)$ and that,
for each integer $i$ with $0\le i<t$, we have a basis
$\uux^{(i)}=(\ux_1^{(i)},\dots,\ux_n^{(i)})$ of $\bZ^n$ such that
\begin{itemize}
 \item[$1)$] $A^{(i)}_j \le \norm{\ux^{(i)}_j} \le 2A^{(i)}_j$
 \quad\text{for $j=1,\dots,n$,}\\
 \item[$2)$] $\disp \norm{\ux^{(i)}_1\wedge\cdots\wedge\widehat{\ux^{(i)}_m}
     \wedge\cdots\wedge\ux^{(i)}_n}
 \ge \frac{\theta}{2} \norm{\ux^{(i)}_1}\cdots\widehat{\norm{\ux^{(i)}_m}}
     \cdots\norm{\ux^{(i)}_n}$
 \quad for $m=k_i$ and $m=\ell_i$.
\end{itemize}
Suppose further that, when $1\le i<t$, we have
\begin{itemize}
 \item[$3)$] $(\ux^{(i)}_1,\dots,\widehat{\ux^{(i)}_{\ell_i}}, \dots, \ux^{(i)}_n)
   = (\ux^{(i-1)}_1,\dots,\widehat{\ux^{(i-1)}_{k_{i-1}}}, \dots, \ux^{(i-1)}_n)$,
 \item[$4)$] $\ux^{(i)}_{\ell_i}
   \in \ux^{(i-1)}_{k_{i-1}} +
    \langle
   \ux^{(i-1)}_1,\dots,\widehat{\ux^{(i-1)}_{k_{i-1}}}, \dots, \ux^{(i-1)}_{\ell_i}
   \rangle_\bZ$.
\end{itemize}
For each integer $i$ with $-1\le i<t$, let $\uu_i$ be a unit vector orthogonal to the subspace
\[
 U_i:=\begin{cases}
  \langle \ux^{(0)}_1,\dots,\ux^{(0)}_{n-1}\rangle_\bR
   &\text{if\/ $i=-1$,}\\[5pt]
  \langle \ux^{(i)}_1,\dots,\widehat{\ux^{(i)}_{k_i}}, \dots, \ux^{(i)}_n\rangle_\bR
   &\text{if\/ $0\le i<t$.}
  \end{cases}
\]
Then, for any $i,j\in\bZ$ with $-1\le i\le j<t$, we have
\begin{equation}
\label{cons:prop1:eq1}
 \dist(\uu_i,\uu_j)=\dist(U_i,U_j)\le 8\theta^{-2}\exp(-q_{i+1}).
\end{equation}
Finally, suppose that $t=s$.  Then, there is a unit vector $\uu$ in $\bR^n$ such that
\begin{equation}
\label{cons:prop1:eq2}
 \dist(\uu_i,\uu)\le 8\theta^{-2}\exp(-q_{i+1}) \quad
 \text{whenever $-1\le i<s$.}
\end{equation}
For any integers $i$, $j$ with $0\le i<s$ and $1\le j\le n$, and any $q \in [q_i,q_{i+1})$,
we also have
\begin{align}
 &-c_1 \le a_j^{(i)}-L_\uu(\ux_j^{(i)},q) \le c_2 &\text{if\/ $j\neq k_i$,}
 \label{cons:prop1:eq3}\\
 &-c_1 \le a_j^{(i)}+q-q_i-L_\uu(\ux_j^{(i)},q) \le c_2 &\text{if\/ $j=k_i$,}
 \label{cons:prop1:eq4}\\
 &L_{\uu,j}(q) \le L_{\uu,j}(\uux^{(i)},q) \le P_j(q)+c_1 \le L_{\uu,j}(q)+c_2
 \label{cons:prop1:eq5}
\end{align}
where $c_1=\log(32/\theta^2)$ and $c_2=nc_1+\log(n!)$.
\end{proposition}

The last property \eqref{cons:prop1:eq5} shows that, for each $q$ in $[q_i,q_{i+1})$,
the basis $\uux^{(i)}$ realizes the minima of the convex body $\cC_\uu(e^q)$ up to
the factor $\exp(c_2)$.  The preceding properties \eqref{cons:prop1:eq3} and
\eqref{cons:prop1:eq4} provide estimates for the individual trajectories of the basis
elements $\ux_j^{(i)}$ over the interval $[q_i,q_{i+1})$, an information which is
partly lost in \eqref{cons:prop1:eq5}.
When $i=s-1<\infty$, we understand the right hand sides of \eqref{cons:prop1:eq1} and
\eqref{cons:prop1:eq2} as $8\theta^{-2}\exp(-\infty)=0$.

\begin{proof}
Fix an integer $i$ with $0\le i<t$.  Using relation 3) when $i\ge 1$ and the
hypothesis that $\ell_0=n$ when $i=0$, we find
\begin{equation}
\label{cons:prop1:proof:eq1}
 U_{i-1}
  =\langle \ux^{(i)}_1,\dots,\widehat{\ux^{(i)}_{\ell_i}}, \dots, \ux^{(i)}_n\rangle_\bR
 \et
 U_i
  = \langle \ux^{(i)}_1,\dots,\widehat{\ux^{(i)}_{k_i}}, \dots, \ux^{(i)}_n\rangle_\bR.
\end{equation}
Since $k_i<\ell_i$, Lemma \ref{dist:lemma:UU} gives
\[
 \dist(U_{i-1},U_i)
 = \frac{\norm{\ux^{(i)}_1\wedge\cdots\wedge\widehat{\ux^{(i)}_{k_i}}
    \wedge\cdots\wedge\widehat{\ux^{(i)}_{\ell_i}}\wedge\cdots\wedge\ux^{(i)}_n}}%
   {\norm{\ux^{(i)}_1\wedge\cdots\wedge\widehat{\ux^{(i)}_{k_i}}
    \wedge\cdots\wedge\ux^{(i)}_n}\,
   \norm{\ux^{(i)}_1\wedge\cdots\wedge\widehat{\ux^{(i)}_{\ell_i}}
    \wedge\cdots\wedge\ux^{(i)}_n}}.
\]
To get an upper bound for this ratio, we apply Hadamard's inequality on the numerator
and hypothesis 2) on each factor of the denominator.  Together with estimates 1),
this yields
\[
 \dist(U_{i-1},U_i)
 \le \frac{4\theta^{-2}}{\norm{\ux^{(i)}_1}\cdots\norm{\ux^{(i)}_n}}
 \le \frac{4\theta^{-2}}{A^{(i)}_1\cdots A^{(i)}_n}
  = 4\theta^{-2}\exp(-q_i).
\]
By Lemma \ref{dist:lemma:uU} and the fact that the distance satisfies the triangle inequality,
we deduce that, for any integers $i$, $j$ with $-1\le i<j<t$, we have
\begin{align*}
 \dist(U_{i},U_j)
 = \dist(\uu_{i},\uu_j)
 &\le \sum_{m=i+1}^j\dist(\uu_{m-1},\uu_m)\\
 &\le 4\theta^{-2}\sum_{m=i+1}^j\exp(-q_m)
  \le 8\theta^{-2}\exp(-q_{i+1}),
\end{align*}
where the last estimate uses the inequality $q_m\ge q_{m-1}+\log(2)$
for each integer $m$ with $1\le m<s$ coming from the hypothesis that the canvas
has mesh $c\ge \log(2)$.  This proves \eqref{cons:prop1:eq1} when $i<j$.  For $i=j$,
this inequality is automatic (it is even sharp when $i=j=s-1<\infty$ since in that case
$\exp(-q_{i+1})=\exp(-\infty)=0$.)

From now on, suppose that $t=s$.  If $s<\infty$, condition \eqref{cons:prop1:eq2}
is fulfilled with $\uu=\uu_{s-1}$ by \eqref{cons:prop1:eq1}.  If $s=\infty$,
then \eqref{cons:prop1:eq1} shows that the image of $(\uu_i)_{i\ge -1}$ in the
projective space over $\bR^n$ is a Cauchy sequence.  So, it converges to the
class of a unit vector $\uu$ of $\bR^n$ which satisfies \eqref{cons:prop1:eq2}.

Finally, fix an integer $i$ with $0\le i<s$, a number $q\in[q_i,q_{i+1})$, and
an index $j\in\{1,\dots,n\}$.  When $j\neq k_i$, we have $\ux_j^{(i)}\in U_i$, thus
$\ux_j^{(i)}\cdot\uu_i=0$ and Lemma \ref{dist:lemma1} yields
\[
 |\ux_j^{(i)}\cdot\uu| \le 2\norm{\ux_j^{(i)}}\dist(\uu_i,\uu).
\]
Using \eqref{cons:prop1:eq2} and the fact that $q<q_{i+1}$, we deduce that
\[
 q+\log|\ux_j^{(i)}\cdot\uu|
   \le c_0+\log\norm{\ux_j^{(i)}}-q_{i+1}+q
   \le c_0+\log\norm{\ux_j^{(i)}}
\]
where $c_0=\log(16/\theta^{2})$.  So, using hypothesis 1), we obtain
\[
 L_\uu(\ux_j^{(i)},q) \le c_0+\log\norm{\ux_j^{(i)}}
    \le c_1+a^{(i)}_j,
\]
which yields the left inequality in \eqref{cons:prop1:eq3}.
When $j=k_i$, we instead have $\ux_{k_i}^{(i)}\in U_{i-1}$ since
$k_i<\ell_i$ (even for $i=0$).   This means that  $\ux_{k_i}^{(i)}\cdot\uu_{i-1}=0$
and so
\[
 |\ux_{k_i}^{(i)}\cdot\uu| \le 2\norm{\ux_{k_i}^{(i)}}\dist(\uu_{i-1},\uu).
\]
Using \eqref{cons:prop1:eq2}, we deduce that
\[
 q+\log|\ux_{k_i}^{(i)}\cdot\uu|
   \le c_0+\log\norm{\ux_{k_i}^{(i)}}-q_{i}+q
\]
for $c_0$ as above.  Then, using the inequality $q\ge q_i$ and hypothesis 1), we obtain
\[
 L_\uu(\ux_{k_i}^{(i)},q) \le c_0+\log\norm{\ux_{k_i}^{(i)}} +q-q_i
    \le c_1+a^{(i)}_{k_i}+q-q_i,
\]
which gives the left inequality in \eqref{cons:prop1:eq4}.  By the above, the points
\[
 \ua=\ua^{(i)}+(q-q_i)\ue_{k_i}+c_1(1,\dots,1)
 \et
 \ub=(L_\uu(\ux^{(i)}_1,q),\dots,L_\uu(\ux^{(i)}_n,q))
\]
satisfy $\ub\le \ua$ for the componentwise ordering.  Since the map $\Phi$ is order
preserving, we deduce that $\Phi(\ub)\le \Phi(\ua)$ and so
\begin{equation}
\label{cons:prop1:proof:eq2}
 \uL_\uu(q) \le \uL_\uu(\uux^{(i)},q)=\Phi(\ub)\le \Phi(\ua)=\uP(q)+c_1(1,\dots,1).
\end{equation}
Since the coordinates of $\ua-\ub$ are non-negative, they are bounded above by their sum
$\Delta$ which is also the sum of the coordinates of $\Phi(\ua)-\Phi(\ub)$.  By
\eqref{cons:prop1:proof:eq2}, this is in turn bounded above by the sum of the coordinates
of $\uP(q)+c_1(1,\dots,1)-\uL_\uu(q)$ which is
\[
 \Delta':=\sum_{j=1}^n (P_j(q)+c_1-L_{\uu,j}(q)) = q + nc_1 - \sum_{j=1}^nL_{\uu,j}(q).
\]
However, Minkowski's second convex body theorem \cite[Chapter IV, Theorem 1A]{Sc1980}
applied to the convex body $\cC_\uu(e^q)$ gives
\[
 L_{\uu,1}(q)+\cdots+L_{\uu,n}(q) \ge \log(2^n/n!) -\log(\vol(\cC_\uu(e^q))) \ge q-\log(n!)
\]
using the crude upper bound $\vol(\cC_\uu(e^q)) \le 2^ne^{-q}$ for the volume of
$\cC_\uu(e^q)$.  So, we obtain that $\Delta\le \Delta'\le nc_1+\log(n!)=c_2$.
From this, we deduce that
$\ua\le \ub+c_2(1,\dots,1)$ which yields the right inequalities in \eqref{cons:prop1:eq3}
and \eqref{cons:prop1:eq4}.  It also gives
\[
 \uP(q)+c_1(1,\dots,1) \le \uL_\uu(q) + c_2(1,\dots,1),
\]
which, together with \eqref{cons:prop1:proof:eq2}, translates into \eqref{cons:prop1:eq5}.
\end{proof}

We will also need the following complement of information.

\begin{cor}
\label{cons:prop1:cor}
Suppose that Proposition \ref{cons:prop1} holds with $t=s$.  Then for each
integer $i$ with $0\le i<s$ such that $q_{i+1}>q_i+\log(2)+c_2$, we have
\begin{equation}
 \left|\log|\ux_{k_i}^{(i)}\cdot\uu| - a_{k_i}^{(i)} + q_i \right| \le c_2
 \label{param:prop1:cor:eq1}
\end{equation}
\end{cor}

\begin{proof}
For such $i$, we apply \eqref{cons:prop1:eq4} for a choice of $q$ with
$q_i+\log(2)+c_2<q<q_{i+1}$. This gives
\[
 L_\uu(\ux^{(i)}_{k_i},q) > a_{k_i}^{(i)}+\log(2) \ge \log\norm{\ux^{(i)}_{k_i}},
\]
where the second inequality comes from condition 1) with $j=k_i$.  Thus,
by definition, we must have $L_\uu(\ux^{(i)}_{k_i},q)
=q+\log|\ux^{(i)}_{k_i}\cdot\uu|$, and \eqref{param:prop1:cor:eq1}
follows by substituting this expression into \eqref{cons:prop1:eq4}.
\end{proof}

In the constructions of \cite[Section 5]{Ro2015}, condition 2) in Proposition
\ref{cons:prop1} is fulfilled by asking that the sequences
$(\ux^{(i)}_1,\dots,\widehat{\ux^{(i)}_m},\dots,\ux^{(i)}_n)$ with $m=k_i$
or $m=\ell_i$ are almost orthogonal in the sense of \cite[Definition 4.5]{Ro2015}.
Here, we use a simpler but more narrow approach.

We fix a linearly independent $(n-1)$-tuple $\uuv=(\uv_1,\dots,\uv_{n-1})$ of
unit vectors of $\bR^n$ and, as in the statement of Theorem \ref{param:thm},
we denote by $(\uuv^{(i)})_{i\ge 0}$ the coherent sequence of directions for
$\uP$ attached to $\uuv$.  We also fix a parameter $\delta$ with
\begin{equation}
\label{cons:delta}
 0 < \delta \le \theta/(4n)
 \quad\text{where}\quad
 \theta=\norm{\uv_1\wedge\cdots\wedge\uv_{n-1}}=\Theta(\uuv),
\end{equation}
using notation \eqref{dist:Theta}.
Then, for each integer $i$ with $0\le i<t$, we ask that the basis $\uux^{(i)}$ satisfies
\begin{equation}
\label{cons:dxv}
 \dist(\ux_j^{(i)},\uv_j^{(i)}) \le \delta \quad \text{for $j=1,\dots,n$.}
\end{equation}
This is stronger than condition 2) of Proposition
\ref{cons:prop1} because, for $m=k_i$ or $m=\ell_i$, the $(n-1)$-tuple
$(\uv^{(i)}_1,\dots,\widehat{\uv^{(i)}_m},\dots,\uv^{(i)}_n)$ is a permutation
of $\uuv$.  So, if \eqref{cons:dxv} holds, then Lemma \ref{dist:lemma2} gives
\[
 \Theta(\ux^{(i)}_1,\dots,\widehat{\ux^{(i)}_m},\dots,\ux^{(i)}_n)
 \ge \Theta(\uuv)-2(n-1)\delta \ge\theta/2.
\]

We will show that the hypotheses of Theorem \ref{param:thm} allow us to
construct recursively a sequence of bases $(\uux^{(i)})_{0\le i<s}$ of $\bZ^n$
that satisfy conditions 1), 3), 4) of Proposition \ref{cons:prop1}, as well as
\eqref{cons:dxv} in replacement of condition 2).  Then this sequence and
the unit vector $\uu$ of $\bR^n$ provided by the proposition possess all the
required properties. Indeed, conditions 3) and 4) mean that the
sequence $(\uux^{(i)})_{0\le i<s}$ is coherent with $\uP$, while condition 1)
yields \eqref{param:thm:eq2}.  The other properties \eqref{param:thm:eq1},
\eqref{param:thm:eq3} and \eqref{param:thm:eq4} follow respectively
from \eqref{cons:dxv}, Corollary \ref{cons:prop1:cor} and
\eqref{cons:prop1:eq5}.

We start by constructing a basis $\uux^{(0)}$ of
$\bZ^n$ which satisfies the hypotheses of
Proposition \ref{cons:prop1} for $t=1$ and conditions
\eqref{cons:dxv} for $i=0$.

\begin{lemma}
\label{cons:lemma1}
If $A_1^{(0)}$ is large enough, with a lower bound
depending only on $\delta$ and $\uuv$, then there is a basis
$\uux^{(0)}=(\ux_1^{(0)},\dots,\ux_n^{(0)})$ of $\bZ^n$ such that
\[
 A_j^{(0)} \le \norm{\ux_j^{(0)}} \le 2A_j^{(0)}
 \et
 \dist(\ux_j^{(0)},\uv_j^{(0)})\le\delta/2
\]
for each $j=1,\dots,n$.  If $\uuv=(\ue_1,\dots,\ue_{n-1})$, it suffices to have
$A_1^{(0)}\ge 6/\delta$.
\end{lemma}

\begin{proof}
By Lemma \ref{red:lemma1}, there is a basis $(\ux_1,\dots,\ux_n)$
of $\bZ^n$ such that $\dist(\ux_j,\uv_j)\le \delta/6$ for each $j=1,\dots,n-1$.
Let $A=\max\{\norm{\ux_1},\dots,\norm{\ux_{n}}\}$ and suppose that
$A_1^{(0)}\ge 6A/\delta$.  For $j=1,\dots,n-1$, we define
\[
 \ux^{(0)}_j=a_j\ux_j+\ux_{j+1}
\]
where $a_j$ is the smallest non-negative integer for which $\norm{\ux_j^{(0)}}
\ge A_j^{(0)}$.  Then $(\ux_1,\ux^{(0)}_1,\dots,\ux^{(0)}_{n-1})$ is a basis of
$\bZ^n$ and, since $k_0<n$, we obtain another basis
$\uux^{(0)}=(\ux_1^{(0)},\dots,\ux_n^{(0)})$ of $\bZ^n$ by setting
\[
  \ux^{(0)}_n=a\ux^{(0)}_{k_0}+\ux_1
\]
where $a$ is the smallest non-negative integer for which $\norm{\ux_n^{(0)}}
\ge A_n^{(0)}$.

For $j=1,\dots,n-1$, we have $A_j^{(0)}\ge A_1^{(0)}\ge 6A$,
thus $a_j\ge 1$ and so
\[
 \norm{\ux_j^{(0)}}\le A_j^{(0)}+A \le 2A_j^{(0)}.
\]
Since $\ux^{(0)}_j\wedge\ux_j=\ux_{j+1}\wedge\ux_j$, we also find
\[
 \dist(\ux^{(0)}_j,\ux_j)
 =\frac{ \norm{\ux_{j+1}\wedge\ux_j} }{ \norm{\ux^{(0)}_j}\,\norm{\ux_j} }
 \le \frac{ \norm{\ux_{j+1}} }{ \norm{\ux^{(0)}_j} }
 \le \frac{A}{A^{(0)}_1}
 \le \frac{\delta}{6}.
\]
So the triangle inequality yields
\[
 \dist(\ux^{(0)}_j,\uv^{(0)}_j)
  = \dist(\ux^{(0)}_j,\uv_j)
 \le \dist(\ux^{(0)}_j,\ux_j) + \dist(\ux_j,\uv_j)
 \le \delta/6 + \delta/6 =\delta/3.
\]
Similarly, since $A^{(0)}_n \ge 2A^{(0)}_{k_0}\ge 12A$, the integer $a$ is positive and so
\[
 \norm{\ux_n^{(0)}}\le A_n^{(0)}+2A_{k_0}^{(0)} \le 2A_n^{(0)}.
\]
Arguing as above, we also find
\[
 \dist(\ux^{(0)}_n,\ux^{(0)}_{k_0})
 \le \frac{ \norm{\ux_1} }{ \norm{\ux^{(0)}_n} }
 \le \frac{A}{A^{(0)}_n}
 \le \frac{\delta}{12},
\]
and so
\begin{align*}
 \dist(\ux^{(0)}_n,\uv^{(0)}_n)
  = \dist(\ux^{(0)}_n,\uv_{k_0})
 &\le \dist(\ux^{(0)}_n,\ux^{(0)}_{k_0}) + \dist(\ux^{(0)}_{k_0},\uv_{k_0})\\
 &\le \delta/12 + \delta/3 <\delta/2.
\end{align*}
If $\uuv=(\ue_1,\dots,\ue_{n-1})$, we can choose $(\ux_1,\dots,\ux_n) =
(\ue_1,\dots,\ue_n)$.  Then, we have $A=1$, and the above construction
requires only that $A_1^{(0)}\ge 6/\delta$.
\end{proof}

In view of the comments which precede Lemma \ref{cons:lemma1}, the next
two propositions complete the proof of Theorem \ref{param:thm},
depending on which of the two conditions \eqref{param:thm:cond1} or
\eqref{param:thm:cond2} holds.

\begin{proposition}
\label{cons:prop:cas1}
Suppose that condition \eqref{param:thm:cond1} holds, namely that
\begin{equation*}
 c\ge \log(8/\delta) \et \ell_i=n \quad\text{whenever $0\le i<s$,}
 \label{cons:prop:cas1:eq1}
\end{equation*}
and that $A_1^{(0)}$ is large enough as a function of $\uuv$ and $\delta$.
Then, we can construct recursively a sequence of bases $(\uux^{(i)})_{0\le i<s}$
of $\bZ^n$ which, for each $i$, satisfies conditions 1) to 4) of Proposition
\ref{cons:prop1} as well as \eqref{cons:dxv}.
If $\uuv=(\ue_1,\dots,\ue_{n-1})$, it suffices to have $A_1^{(0)}\ge 6/\delta$.
\end{proposition}

\begin{proof}
If $A_1^{(0)}$ is large enough, Lemma \ref{cons:lemma1} provides a basis
$\uux^{(0)}$ which satisfies these conditions for $i=0$.  If $s=1$, we are done.
Otherwise, suppose that we have constructed appropriate bases
$\uux^{(0)},\dots,\uux^{(t-1)}$ of $\bZ^n$ for some integer $t$ with $1\le t<s$.
For each $j=1,\dots,n-1$, we have $\ux_j^{(0)}\in U_{-1}$ and so
\[
 \dist(\uv_j,U_{-1})
  = \dist(\uv_j^{(0)},U_{-1})
  \le \dist(\uv_j^{(0)},\ux_j^{(0)})\le \delta/2.
\]
Using \eqref{cons:prop1:eq1} and the fact that $q_0\ge 2a_1^{(0)}$,
we also find
\[
 \dist(U_{-1},U_{t-1})
  \le 8\theta^{-2}\exp(-q_0)\le 8\theta^{-2}(A_1^{(0)})^{-2}\le \delta/{4}
\]
if $A_1^{(0)}\ge 6/(\theta\delta)$.  Assuming this, we deduce that
\[
 \dist(\uv_j,U_{t-1})
  \le \dist(\uv_j,U_{-1})+\dist(U_{-1},U_{t-1})\le 3\delta/4
\]
for $j=1,\dots,n-1$.  Finally, since  $\uv_n^{(t)}$ is one of the points
$\uv_1,\dots,\uv_{n-1}$, we conclude that there is a unit vector $\uv\in U_{t-1}$
for which
\begin{equation}
 \dist(\uv_n^{(t)},\uv) =\dist(\uv_n^{(t)},U_{t-1}) \le 3\delta/4.
 \label{cons:prop:cas1:proof:eq3}
\end{equation}

To alleviate notation, we set
\[
 h=k_{t-1}, \quad k=k_t,
\]
and note that, by definition of a canvas, we  have $h<n$ and $k<n$
since $\ell_{t-1}=\ell_t=n$.  We further set
\[
  (\ux^{(t)}_1,\dots, \ux^{(t)}_{n-1})
   = (\ux^{(t-1)}_1,\dots,\widehat{\ux^{(t-1)}_h}, \dots, \ux^{(t-1)}_n),
\]
as prescribed by condition 3) for $i=t$.  As the same relation holds between
$\uuv^{(t)}$ and $\uuv^{(t-1)}$ as well as between $\ua^{(t)}$ and $\ua^{(t-1)}$,
the induction hypothesis implies that
\begin{equation}
  \dist(\ux_j^{(t)},\uv_j^{(t)}) \le \delta
  \et
  A_j^{(t)}\le \norm{\ux_j^{(t)}}\le 2A_j^{(t)}
 \label{cons:prop:cas1:proof:eq4}
\end{equation}
for each $j=1,\dots,n-1$.  To complete the induction step, it remains to construct
\begin{equation}
  \ux^{(t)}_n
   \in \ux^{(t-1)}_h + \langle \ux^{(t-1)}_1,\dots,
         \widehat{\disp\ux_h^{(t-1)}},\dots,\ux^{(t-1)}_{n}\rangle_\bZ
 \label{cons:prop:cas1:proof:eq5}
\end{equation}
so that \eqref{cons:prop:cas1:proof:eq4} holds as well for $j=n$, for
then $\uux^{(t)}=(\ux_1^{(t)},\dots,\ux_n^{(t)})$
is a basis of $\bZ^n$ with the required properties.

In view of \eqref{cons:prop:cas1:proof:eq3}, we simply need to construct
$\ux^{(t)}_n$ so that it fulfills \eqref{cons:prop:cas1:proof:eq5} as well as
\begin{equation}
  \dist(\ux_n^{(t)},\uv) \le \delta/4
  \et
  A\le \norm{\ux_n^{(t)}}\le 2A
 \quad\text{where}\quad A=A_n^{(t)}.
 \label{cons:prop:cas1:proof:eq6}
\end{equation}
Since $\disp (\ux_1^{(t-1)},\dots,\ux_n^{(t-1)})$ is a basis of $\bZ^n$, we find
\[
 \ux_h^{(t-1)}-r_0\uu_{t-1} \in U_{t-1}
 \quad\text{where}\quad
 |r_0| = \big|\ux_h^{(t-1)}\cdot\uu_{t-1}\big| \le 1
\]
(see \cite[lemma 4.1]{Ro2015}).  Since $\uv$ belongs to $U_{t-1}$ as well, we
may therefore write
\begin{equation}
 r_0\uu_{t-1}+\frac{3}{2}A\uv = \ux_h^{(t-1)} + \sum_{j\neq h}r_j\ux_j^{(t-1)}
 \label{cons:prop:cas1:proof:eq7}
\end{equation}
for some coefficients $r_j\in \bR$, where the sum extends to all $j=1,\dots,n$ with $j\neq h$.
For each of those $j$, we choose an integer $a_j$ such that $\abs{a_j-r_j}\le 1/2$.
Then the point
\[
 \ux_n^{(t)} = \ux_h^{(t-1)} + \sum_{j\neq h}a_j\ux_j^{(t-1)}
\]
satisfies \eqref{cons:prop:cas1:proof:eq5}.  Using \eqref{cons:prop:cas1:proof:eq7}
we find
\[
 \Big\| \ux_n^{(t)}-\frac{3}{2}A\uv \Big\|
 = \Big\| r_0\uu_{t-1} + \sum_{j\neq h}(a_j-r_j)\ux_j^{(t-1)} \Big\|
 \le 1 + \frac{1}{2}\sum_{j\neq h}\norm{\ux_j^{(t-1)}}.
\]
Since condition 1) holds for $i=t-1$, this yields
\begin{align*}
 \Big\| \ux_n^{(t)}-\frac{3}{2}A\uv \Big\|
   \le 1 + \sum_{j\neq h} A_j^{(t-1)}
   \le \sum_{j=1}^n A_j^{(t-1)}
   \le 2A_n^{(t-1)}
   \le \frac{A\delta}{4},
\end{align*}
where the second and third inequalities use \eqref{cons:eq:A} with $i=t-1$ together
with the hypothesis that $\exp(c)\ge 8/\delta\ge 8$, while the
last inequality uses $A_n^{(t-1)}=A_{n-1}^{(t)}\le \exp(-c)A_n^{(t)}$.
As $\delta\le 1$, this implies that $A\le\norm{\ux_n^{(t)}}\le 2A$ and
\[
  \dist(\ux_n^{(t)},\uv)
  =\frac{\norm{(\ux_n^{(t)}-(3/2)A\uv)\wedge\uv}}{\norm{\ux_n^{(t)}}}
 \le \frac{\norm{\ux_n^{(t)}-(3/2)A\uv}}{\norm{\ux_n^{(t)}}}
 \le \frac{A\delta/4}{A}=\frac{\delta}{4}
\]
as requested by \eqref{cons:prop:cas1:proof:eq6}.  This completes the recursion step.
The last assertion of the proposition is easily checked.
\end{proof}

While the above argument is close to that in \cite[Section 5]{Ro2015}, the next
statement uses an even simpler construction of points.

\begin{proposition}
\label{cons:prop:cas2}
Suppose that condition \eqref{param:thm:cond2} holds, namely that we have
$c\ge \log(2)$ and
\begin{equation}
\label{cons:prop:cas2:eq1}
 \sum_{i=1}^s \exp(q_{i-1}-q_i) \le \frac{\delta}{4},
\end{equation}
and suppose that $A_1^{(0)}$ is large enough as a function of $\uuv$ and $\delta$.
Then, we can construct recursively a sequence of bases $(\uux^{(i)})_{0\le i<s}$
of $\bZ^n$ which, for each $i$, satisfies conditions 1) to 4) of Proposition
\ref{cons:prop1} as well as the constraints
\begin{equation}
\label{cons:prop:cas2:proof:eq1}
 \dist(\ux_j^{(i)},\uv_j^{(i)}) \le \frac{\delta}{2}+2\sum_{m=1}^i \exp(q_{m-1}-q_m)
 \quad\text{for $j=1,\dots,n$,}
\end{equation}
which, in view of \eqref{cons:prop:cas2:eq1}, are stronger than \eqref{cons:dxv}.
If $\uuv=(\ue_1,\dots,\ue_{n-1})$, it suffices to have $A_1^{(0)}\ge 6/\delta$.
\end{proposition}

\begin{proof}
If $A_1^{(0)}$ is large enough, Lemma \ref{cons:lemma1} provides a
basis $\uux^{(0)}$ with the required property.  For the recurrence step,
suppose that we have constructed appropriate bases $\uux^{(0)},\dots,\uux^{(t-1)}$
of $\bZ^n$ for some integer $t$ with $1\le t<s$.  To alleviate notation, we set
\[
 h=k_{t-1}, \quad \ell=\ell_t, \quad k=k_t,
\]
and note that, by conditions (C2) and (C3) in Definition \ref{param:def:canvas}
of a canvas, we  have $h\le\ell$ and $k<\ell$.  We also observe that
\begin{equation}
\label{cons:prop:cas2:proof:eq2}
 \frac{A_\ell^{(t)}}{A_h^{(t-1)}} = \exp(q_t-q_{t-1}) \ge \exp(c) \ge 2
\end{equation}
in view of relation \eqref{param:def:canvas:eq} for $i=t-1$, and the formula for the
switch numbers $q_i$ in Definition \ref{param:def:systems}.  Then we define
$\uux^{(t)}=(\ux^{(t)}_1,\dots,\ux^{(t)}_n)$ by
\begin{align}
  (\ux^{(t)}_1,\dots,\widehat{\ux^{(t)}_\ell}, \dots, \ux^{(t)}_n)
   &= (\ux^{(t-1)}_1,\dots,\widehat{\ux^{(t-1)}_h}, \dots, \ux^{(t-1)}_n),
 \label{cons:prop:cas2:proof:eq3}\\
  \ux^{(t)}_\ell &= \ux^{(t-1)}_h + a \ux^{(t)}_k,
  \label{cons:prop:cas2:proof:eq4}
\end{align}
where $a$ is the smallest non-negative integer such that
\begin{equation}
\label{cons:prop:cas2:proof:eq5}
  A_\ell^{(t)} < \norm{\ux^{(t)}_\ell}.
\end{equation}
Since $\uux^{(t-1)}$ is a basis of $\bZ^n$, this $n$-tuple $\uux^{(t)}$ is
also a basis of $\bZ^n$.  By construction, it fulfills  conditions 3) and 4)
for $i=t$.  Moreover, since $\ua^{(t)}$ and $\ua^{(t-1)}$ are linked by the same
relation as $\uux^{(t)}$ and $\uux^{(t-1)}$ in \eqref{cons:prop:cas2:proof:eq3},
and since by hypothesis condition 1) holds for $i=t-1$ and all $j=1,\dots,n$,
that condition also holds for $i=t$ except possibly when $j=\ell$.  In particular,
since $k<\ell$, we obtain
\[
 \norm{\ux^{(t)}_k} \le 2A_k^{(t)}\le A_\ell^{(t)}
\]
where the second inequality uses \eqref{cons:eq:A}.
Since $h\le \ell$, we also find
\[
 \norm{\ux^{(t-1)}_h} \le 2A_h^{(t-1)} \le A_\ell^{(t)}
\]
using \eqref{cons:prop:cas2:proof:eq2}.  Thus the integer $a$ must be
positive, and so
\[
 \norm{\ux^{(t)}_\ell}
  \le A_\ell^{(t)} + \norm{\ux^{(t)}_k}
  \le 2A_\ell^{(t)}.
\]
Together with \eqref{cons:prop:cas2:proof:eq5}, this shows that condition 1)
holds as well for $i=t$ and $j=\ell$.  Similarly, since inequality
\eqref{cons:prop:cas2:proof:eq1}
holds for $i=t-1$ and all $j=1,\dots,n$, and since $\uuv^{(t)}$ and $\uuv^{(t-1)}$
are linked in the same way as $\uux^{(t)}$ and $\uux^{(t-1)}$
in \eqref{cons:prop:cas2:proof:eq3}, that inequality
also holds for $i=t$ and $j\neq \ell$ in the stronger form
\begin{equation}
\label{cons:prop:cas2:proof:eq6}
 \dist(\ux^{(t)}_j,\uv^{(t)}_j)
  \le \frac{\delta}{2}+2\sum_{m=1}^{t-1}\exp(q_{m-1}-q_m).
\end{equation}
To estimate this distance when $j=\ell$, we first note, using
\eqref{cons:prop:cas2:proof:eq4}, that
\[
 \norm{\ux_\ell^{(t)}\wedge\ux_k^{(t)}}
 = \norm{\ux_h^{(t-1)}\wedge\ux_k^{(t)}}
 \le \norm{\ux_h^{(t-1)}}\norm{\ux_k^{(t)}},
\]
and so
\[
 \dist(\ux^{(t)}_\ell,\ux^{(t)}_k)
  \le \frac{\norm{\ux^{(t-1)}_h}}{\norm{\ux^{(t)}_\ell}}
  \le \frac{2A_h^{(t-1)}}{A_\ell^{(t)}}
   = 2\exp(q_{t-1}-q_t).
\]
Together with \eqref{cons:prop:cas2:proof:eq6} for $j=k$, this yields
\[
 \dist(\ux^{(t)}_\ell,\uv^{(t)}_k)
  \le \dist(\ux^{(t)}_\ell,\ux^{(t)}_k) +  \dist(\ux^{(t)}_k,\uv^{(t)}_k)
  \le \frac{\delta}{2}+2\sum_{m=1}^{t}\exp(q_{m-1}-q_m).
\]
Since $\uv^{(t)}_\ell=\uv^{(t)}_k$, this shows that \eqref{cons:prop:cas2:proof:eq1}
holds for $i=t$. Thus condition 2) holds as well for $i=t$.
The last assertion of the proposition is clear from the statement of Lemma \ref{cons:lemma1}.
\end{proof}

%
%

\section{Application to a specific family of generalized $n$-systems}
\label{sec:app}

The goal of this section is to apply Theorem \ref{param:thm} on parametric geometry
of numbers with constraints to produce points $\uu$ that satisfy the second part of
Theorem \ref{intro:thm1}, thereby completing the proof of the latter theorem.
We first construct a generalized $n$-system $\tuP$ in the sense of \cite[\S 4]{Ro2016},
and we approximate it by a rigid $n$-system $\uP$ to which Theorem \ref{param:thm}
applies. Then, we use geometry of numbers to show that the point $\uu$ provided by
the latter theorem has the required property.

To this end, we fix integers $m\ge 1$ and $n\ge m+2$, and set
\[
r=n-m.
\]
We also fix an orthonormal basis $(\uv_1,\dots,\uv_n)$ of $\bR^n$ and set
\[
 \uuv=(\uv_1,\dots,\uv_{n-1}),
\]
so that $\Theta(\uuv)=1$ in the notation of \eqref{dist:Theta}.  We choose real
numbers $\delta$ and $c$ with
\begin{equation}
\label{app:eq:delta,c}
 0<\delta\le 1/(4n) \et  c=\log(8/\delta),
\end{equation}
and we denote by $\kappa$ the constant provided by Theorem \ref{param:thm} for
the above choice of $\uuv$ and $\delta$.  Finally we choose sequences of positive real
numbers $(X_i)_{i\ge 0}$ and $(Y_i)_{i\ge 0}$ with the following properties
\begin{align}
&\kappa+c\le \log(X_0),
 \label{app:eq:X:0}\\
 &c+\log(X_i)\le \log(X_{i+1}) \quad\text{for $i=0,\dots,r-2$,}
 \label{app:eq:X:1}\\
 &(3m+n)c+\log(X_{i+r-1})\le \log(Y_i) \le -2mc+\log(X_{i+r})\quad\text{for each $i\ge 0$.}
 \label{app:eq:XY}
\end{align}
To this data, we attach a continuous piecewise linear map
$\tuP=(\tP_1,\dots,\tP_n)\colon [q_0,\infty)\to\bR^n$
which satisfies
\[
 0\le \tP_1(q)\le\cdots\le\tP_r(q)\le\tP_{r+1}(q)=\dots=\tP_n(q)
 \et
 \tP_1(q)+\cdots+\tP_n(q)=q
\]
for each $q\ge q_0$.  Its combined graph, namely the union of the graphs
of its components $\tP_1,\dots,\tP_n$, is given by Figure \ref{fig1}.

\begin{figure}[ht]
     \begin{tikzpicture}[scale=0.46]
       \draw[dashed] (2,11)--(2,2) node[below]{$q_i$};
       \draw[dashed] (10,11)--(10,2) node[below]{$s_i$};
       \draw[dashed] (18,13)--(18,2) node[below]{$t_i$};
       \draw[dashed] (26,17)--(26,2) node[below]{$q_{i+1}$};
       \draw[thick] (2,3)--(10,11)--(18,13)--(26,17);
       \draw[thick] (2,5) -- (26,5);
       \draw[thick] (2,9) -- (26,9);
       \draw[thick] (2,11) -- (10,11);
       \draw[thick] (2,11) -- (10,11);
       \draw[thick] (18,13) -- (26,13);
       \node[draw,circle,inner sep=1pt,fill] at (10,11) {};
       \node[draw,circle,inner sep=1pt,fill] at (18,13) {};
       \node[left] at (1.5,3) {$\log(X_{i})$};
       \node[left] at (1.5,5) {$\log(X_{i+1})$};
       \node[right] at (26.5,5) {$\log(X_{i+1})$};
       \node[left] at (0.5,7) {$\cdots$};
       \node[right] at (27.5,7) {$\cdots$};
       \node[left] at (1.5,9) {$\log(X_{i+r-1})$};
       \node[right] at (26.5,9) {$\log(X_{i+r-1})$};
       \node[left] at (1.5,11) {$\log(Y_i)$};
       \node[right] at (26.5,13) {$\log(X_{i+r})$};
       \node[right] at (26.5,17) {$\log(Y_{i+1})$};
      \node[above] at (6,11) {$\tP_{r+1}=\cdots=\tP_n$};
      \node[above, right] at (10,14.8) {$\tP_{r}=\cdots=\tP_n$};
      \node[above,right] at (10,13.5) {of slope $1/(m+1)$};
      \node[above, right] at (18,17.5) {$\tP_{r+1}=\cdots=\tP_n$};
      \node[above,right] at (18,16.2) {of slope $1/m$};
      \node[above] at (14,5) {$\tP_1$};
      \node[above] at (14,7) {$\cdots$};
      \node[above] at (14,9) {$\tP_{r-1}$};
      \node[above] at (22,13) {$\tP_r$};
       \end{tikzpicture}
\caption{The combined graph of a generalized $n$-system.}
\label{fig1}
\end{figure}
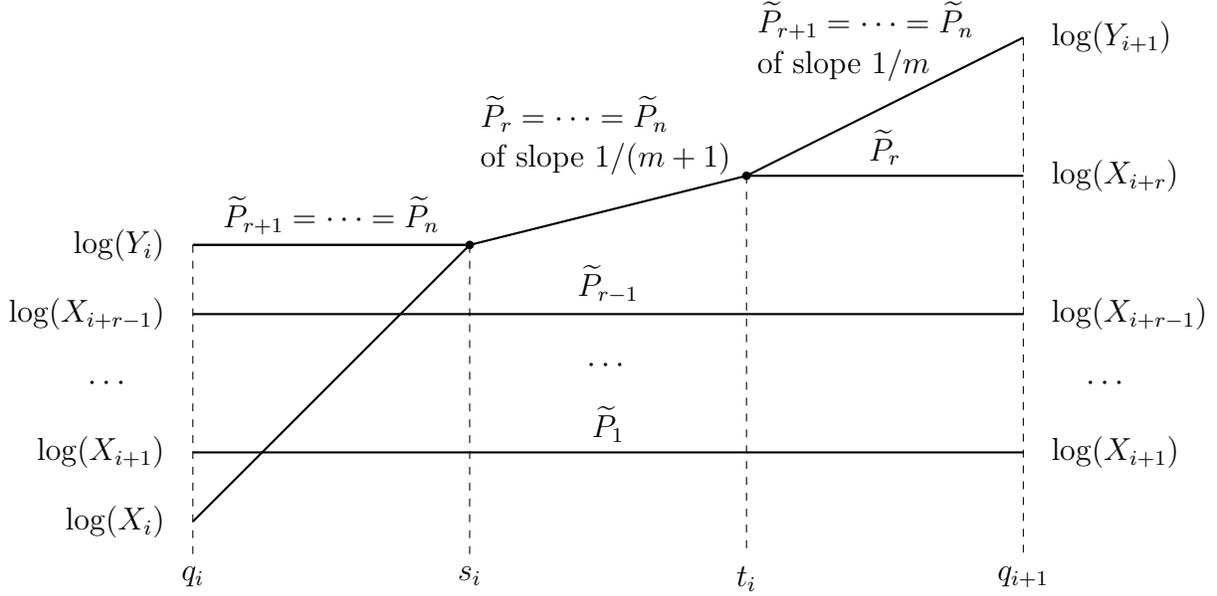

Explicitly, we construct sequences $(q_i)_{i\ge 0}$, $(s_i)_{i\ge 0}$ and $(t_i)_{i\ge 0}$
by setting
\begin{align*}
 &q_i = \log(X_iX_{i+1}\cdots X_{i+r-1}Y_i^m),\\
 &s_i = \log(X_{i+1}\cdots X_{i+r-1} Y_i^{m+1}),\\
 &t_i = \log(X_{i+1}\cdots X_{i+r-1} X_{i+r}^{m+1})
\end{align*}
for each $i\ge 0$.  Then, we have $q_i<s_i<t_i<q_{i+1}$ and, as illustrated
in Figure \ref{fig1}, we define $\tP_1,\dots,\tP_n$ on $[q_i,q_{i+1}]$ by
\begin{align*}
 \tP_{r+1}(q)=\cdots=\tP_n(q)
  &=\begin{cases}
    \disp \log(Y_i) &\text{if $q_i\le q\le s_i$,}\\[6pt]
    \disp \log(Y_i) +\frac{q-s_i}{m+1} &\text{if $s_i\le q\le t_i$,}\\[10pt]
    \disp \log(X_{i+r}) +\frac{q-t_i}{m} &\text{if $t_i\le q\le q_{i+1}$,}
    \end{cases}\\
(\tP_1(q),\dots,\tP_r(q))
    &=\begin{cases}
         \Phi_r(\log(X_{i+1}),\dots,\log(X_{i+r-1}),q-q_i+\log(X_i))
            &\text{if $q_i\le q\le s_i$,}\\[6pt]
         (\log(X_{i+1}),\dots,\log(X_{i+r-1}),\tP_{r+1}(q))
            &\text{if $s_i\le q\le t_i$,}\\[6pt]
         (\log(X_{i+1}),\dots,\log(X_{i+r})) &\text{if $t_i\le q\le q_{i+1}$,}
      \end{cases}
\end{align*}
where $\Phi_r\colon\bR^r\to\Delta_r\subset\bR^r$ is the ordering map defined in Section \ref{sec:param}.

In the terminology of \cite{Ro2016}, the map $\tuP$ is a generalized $n$-system and
Section 4 of that paper provides a general method to approximate such a map by an
ordinary $n$-system $\uP$.  To prove the next result, we adapt this method to
produce an approximate rigid $n$-system $\uP$ whose transition indices $\ell_j$ are all
equal to $n$, so that Theorem \ref{param:thm} applies to it.

\begin{theorem}
\label{app:thm}
Let $(\tuv_i)_{i\ge 0}$ denote the periodic sequence of period $r$ with
$\tuv_i=\uv_{i+1}$ for each $i=0,\dots,r-1$.  For the above data, there exist
a unit vector $\uu$ of $\bR^n$ whose coordinates are linearly independent over
$\bQ$, and a sequence of non-zero points $(\ux_i)_{i\ge 0}$ in $\bZ^n$ such that,
for each $i\ge 0$ and each $q\ge q_0$, we have
\begin{itemize}
\item[$1)$] $\dist(\ux_i,\tuv_i)\le \delta$,
\medskip
\item[$2)$] $\abs{\log\norm{\ux_i}-\log(X_i)}\le nc$
   \ and \ $\abs{\log\abs{\ux_i\cdot\uu} + q_i -\log(X_i)}\le 4mnc$,
\medskip
\item[$3)$] $\norm{\tuP(q)-\uL_\uu(q)}_\infty \le 5mnc$,
\end{itemize}
where $\norm{\ }_\infty$ stands for the maximum norm on $\bR^n$.
\end{theorem}

\begin{proof}
Our construction of an approximate rigid $n$-system $\uP$ differs slightly
depending on whether $m=1$ or  $m>1$.  To cover both cases, we set
\begin{equation}
  m^*=\max\{1,m-1\}.
  \label{app:eq:m*}
\end{equation}
We first define two sequences $(a_i)_{i\ge 0}$ and $(b_i)_{i\ge 0}$
of multiples of $c$ satisfying
\begin{equation}
  -mc < a_i - \log(X_i) \le 0
  \et
  -m^*c < b_i + c - \log(Y_i) \le 0
  \label{app:eq:ab}
\end{equation}
for each $i\ge 0$.  For $a_0,\dots,a_{r-1}$ and $b_0$, we choose the largest
multiples of $c$ satisfying these conditions.  Then, for each $i\ge 0$, we define
recursively
\begin{equation}
  a_{i+r}=b_i+\sigma(i)mc \et b_{i+1}=a_{i+r}+\tau(i)m^*c+c
  \label{app:def:ab}
\end{equation}
with integers $\sigma(i)$ and $\tau(i)$ chosen so that $a_{i+r}$ and
$b_{i+1}$ satisfy conditions \eqref{app:eq:ab}.
Hypotheses
\eqref{app:eq:X:0} and \eqref{app:eq:X:1} imply that $0<a_0<\cdots<a_{r-1}$
while \eqref{app:eq:XY} yields
\begin{align}
 b_i-a_{i+r-1}
  &\ge \log(Y_i)-(m^*+1)c-\log(X_{i+r-1}) \ge (n+m)c,
  \label{app:thm:eq:ba}\\
 a_{i+r}-b_i
  &\ge \log(X_{i+r})-mc-\log(Y_i) \ge mc
 \label{app:thm:eq:ab}
\end{align}
for each $i\ge 0$.  In particular, we have $\sigma(i)\ge 1$ and $\tau(i)\ge 2$
for each $i\ge 0$.

For each pair of integers $(i,j)$ with $i\ge 0$ and
\begin{equation}
 0\le j\le \nu(i):=\sigma(i)m+\tau(i)(m-1),
  \label{app:eq:j}
\end{equation}
we define
\[
\begin{array}{ll}
 \begin{aligned}
   &\ua^{(i,j)}=(a_i,\dots,a_{i+r-1},b_i-(m-1)c,\dots,b_i),\\
   &k_{i,j}=1, \hspace*{32pt}\ell_{i,j}=n
  \end{aligned}
  &\Big\}\  \text{if $j=0$,}\\[15pt]
 \begin{aligned}
   &\ua^{(i,j)}=(a_{i+1},\dots,a_{i+r-1},b_i+(j-m)c,\dots,b_i+jc),\\
   &k_{i,j}=r, \hspace*{32pt} \ell_{i,j}=n
  \end{aligned}
  &\Big\}\  \text{if $0 < j < \sigma(i)m+m$,}\\[15pt]
 \begin{aligned}
   &\ua^{(i,j)}=(a_{i+1},\dots,a_{i+r},b_i+(j+1-m)c,\dots,b_i+jc),\\
   &k_{i,j}=r+1, \quad \ell_{i,j}=n
  \end{aligned}
  &\Big\}\  \text{if $j\ge \sigma(i)m+m$,}
\end{array}
\]
and we denote by $q_{i,j}$ the sum of the coordinates of $\ua^{(i,j)}$.  By
\eqref{app:thm:eq:ba}, \eqref{app:thm:eq:ab} and above, each $\ua^{(i,j)}$ is a
strictly increasing sequence of positive multiples of $c$, ending in an
arithmetic progression with difference $c$.  We also note that $\sigma(i)m+m
\le \nu(i)$ if and only if $m>1$, because $\tau(i)\ge 2$.  Thus pairs
$(i,j)$ with $\sigma(i)m+m\le j\le \nu(i)$ occur when $m>1$, but not when
$m=1$. Using the lexicographical ordering in which $(i,j)<(i',j')$ when either $i<i'$
or both $i=i'$ and $j<j'$, it follows that this data defines a canvas
of mesh $c$.  Let $\uP=(P_1,\dots,P_n)\colon[q_{0,0},\infty)\to\Delta_n$
denote its associated $n$-system.  By definition, its sequence of switch numbers is
$(q_{i,j})$.

Figure \ref{fig2} shows the combined graph of $\uP$ on a typical interval
$[q_{i,0},q_{i+1,0}]$ when $m=2$.  For larger $m$, the graph is similar.
However, when $m=1$, it is sensibly different because there is no switch
number of $\uP$ inside the interval $[q_{i,\sigma(i)},\,q_{i+1,0}]$, and the
graph of $\uP$ is affine linear with $P_{r+1}=P_n$ of slope $1$ over that interval.
As the picture illustrates, the behaviour of $\uP$ is similar to that of $\tuP$ and
we will show in fact that their difference $\uP-\tuP$ is bounded.

\begin{figure}[ht]
     \begin{tikzpicture}[scale=0.46]
       \draw[dashed, semithick] (2,11)--(2,2) node[below]{$q_{i,0}$};
       \draw[dashed, semithick] (10,11.8)--(10,2) node[below]{$q_{i,1}$};
       \draw[dashed] (12.4,12.6)--(12.4,3) node[below]{$q_{i,2}$};
       \draw[dashed] (14.8,13.4)--(14.8,3) node[below]{$q_{i,3}$};
       \draw[dashed, semithick] (20.4,15)--(20.4,2) node[below]{$q_{i,2\sigma(i)+2}$};
       \draw[dashed] (22,15.8)--(22,3);
       \node[below] at (22.4,3){$q_{i,2\sigma(i)+3}$};
       \draw[dashed] (18,14.2)--(18,3);
       \node[below] at (18.3,3){$q_{i,2\sigma(i)+1}$};
       \node[below] at (16.4,2){$\cdots$};
       \node[below] at (24.8,2) {$\cdots$};
       \draw[dashed] (27.6,18.2)--(27.6,3);
       \node[below] at (26.8,3){$q_{i,2\sigma(i)+\tau(i)}$};
       \draw[dashed, semithick] (29.2,19)--(29.2,2) node[below]{$q_{i+1,0}$};
       \draw[thick] (2,3)--(10,11.8)--(14.8,11.8)--(15,12);
       \draw[thick] (2,11)--(12.4,11)--(14.8,13.4)--(15,13.4);
       \draw[thick] (2,10.2)--(10,10.2)--(12.4,12.6)--(14.8,12.6)--(15,12.6);
       \draw[thick] (17.8,14)--(18,14.2)--(20.4,14.2)--(22,15.8)--(23.6,15.8);
       \draw[thick] (17.8,13.4)--(18,13.4)--(29.2,13.4);
       \draw[thick] (17.8,12.6)--(18,12.6)--(20.4,15)--(22,15)--(23.6,16.6);
       \draw[thick] (26,16.6)--(27.6,18.2)--(29.2,18.2);
       \draw[thick] (26,17.4)--(27.6,17.4)--(29.2,19);

       \draw[thick] (2,5) -- (29.2,5);
       \draw[thick] (2,8) -- (29.2,8);
       \draw[thick] (2,11) -- (10,11);
       \draw[thick] (2,11) -- (10,11);
       \node[left] at (1.5,3) {$a_{i}$};
       \node[left] at (1.5,5) {$a_{i+1}$};
       \node[right] at (29.7,5) {$a_{i+1}$};
       \node[left] at (0.5,7) {$\vdots$};
       \node[right] at (30.7,7) {$\vdots$};
       \node[left] at (1.5,8) {$a_{i+r-1}$};
       \node[right] at (29.7,8) {$a_{i+r-1}$};
       \node[left] at (1.5,10.2) {$b_i-c$};
       \node[left] at (1.5,11.1) {$b_i$};
       \node[right] at (29.7,13) {$a_{i+r}$};
       \node[right] at (29.7,18.2) {$b_{i+1}-c$};
       \node[right] at (29.7,19.1) {$b_{i+1}$};
      \node[above] at (24.8,5) {$P_1$};
      \node at (24.8,7) {$\vdots$};
      \node[above] at (24.8,8) {$P_{r-1}$};
      \node[below] at (24.8,13.4) {$P_r$};
      \node at (24.8,16.6) {$\dots$};
     \node at (16.4,13) {$\dots$};
       \end{tikzpicture}
\caption{The combined graph of $\uP$ on $[q_{i,0},q_{i+1,0}]$ when $m=2$.}
\label{fig2}
\end{figure}
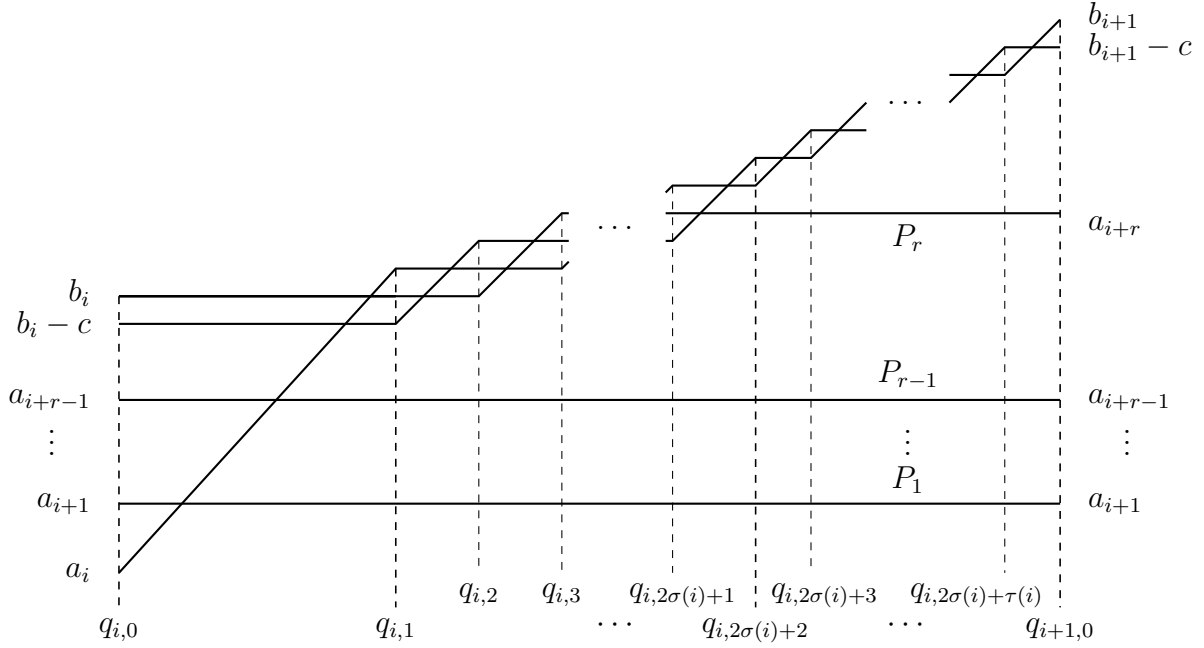

Before we do this, consider the coherent system of directions
$\uuv^{(i,j)}=(\uv_1^{(i,j)},\dots,\uv_n^{(i,j)})$ with $i\ge 0$ and $0\le j\le \nu(i)$
attached to $\uuv=(\uv_1,\dots,\uv_{n-1})$.  We claim that, for each $i\ge 0$,
we have
\begin{equation}
\label{app:claim:v(i,0)}
 \uuv^{(i,0)}=(\tuv_i,\dots,\tuv_{i+r-1},\uv_{r+1},\dots,\uv_{n-1},\tuv_i).
\end{equation}
For $i=0$, this is clear since $k_{0,0}=1$ and thus
$\uuv^{(0,0)}=(\uv_1,\dots,\uv_{n-1},\uv_1)$. Now, suppose that
\eqref{app:claim:v(i,0)} holds for some $i\ge 0$.  If $m=1$, we have $r=n-1\ge 2$
and this formula becomes
\[
 \uuv^{(i,0)}=(\tuv_i,\dots,\tuv_{i+n-2},\tuv_i).
\]
As $k_{i,0}=1$, $k_{i,1}=\dots=k_{i,\sigma(i)}=n-1$ and $k_{i+1,0}=1$, we deduce that
\begin{align*}
 &\uuv^{(i,1)}=\cdots=\uuv^{(i,\sigma(i))}
   =(\tuv_{i+1},\dots,\tuv_{i+n-2},\tuv_i,\tuv_i),\\
&\uuv^{(i+1,0)}=(\tuv_{i+1},\dots,\tuv_{i+n-2},\tuv_i,\tuv_{i+1})
   =(\tuv_{i+1},\dots,\tuv_{i+n-1},\tuv_{i+1}),
\end{align*}
which proves our claim \eqref{app:claim:v(i,0)} by induction on $i$.  If $m>1$, we
have $k_{i,0}=1$ and $k_{i,j}=r$ for $1\le j<\sigma(i)m+m$.  Then, the
sequence $\uuv^{(i,1)}, \uuv^{(i,2)},\dots,\uuv^{(i,\sigma(i)m+m-1)}$
is again periodic of period $m$ with
\begin{align*}
  \uuv^{(i,1)}
    &=(\tuv_{i+1},\dots,\tuv_{i+r-1},\uv_{r+1},\dots,\uv_{n-1},\tuv_i,\uv_{r+1}),\\
  \uuv^{(i,2)}
    &=(\tuv_{i+1},\dots,\tuv_{i+r-1},\uv_{r+2},\dots,\uv_{n-1},\tuv_i,\uv_{r+1},
          \uv_{r+2}),\\
    &\dots\\
  \uuv^{(i,m-1)}
    &=(\tuv_{i+1},\dots,\tuv_{i+r-1},\uv_{n-1},\tuv_i,\uv_{r+1},\dots,\uv_{n-1}),\\
   \uuv^{(i,m)}
    &=(\tuv_{i+1},\dots,\tuv_{i+r-1},\tuv_i,\uv_{r+1},\dots,\uv_{n-1},\tuv_i).
\end{align*}
In particular, we have $\uuv^{(i,\sigma(i)m+m-1)}=\uuv^{(i,m-1)}$.
Since $k_{i,\sigma(i)m+m}=r+1$, we deduce that
\begin{align*}
 \uuv^{(i,\sigma(i)m+m)}
   &= (\tuv_{i+1},\dots,\tuv_{i+r-1},\tuv_i,\uv_{r+1},\dots,\uv_{n-1},\uv_{r+1})\\
   &= (\tuv_{i+1},\dots,\tuv_{i+r},\uv_{r+1},\dots,\uv_{n-1},\uv_{r+1}).
\end{align*}
Finally since $k_{i,j}=r+1$ for $(\sigma(i)+1)m\le j\le \nu(i)$, we find similarly
that the sequence $\uuv^{(i,\sigma(i)m+m)}, \dots,\uuv^{(i,\nu(i))}$
is periodic of period $m-1$ and that
\[
 \uuv^{(i,\nu(i))}
   = (\tuv_{i+1},\dots,\tuv_{i+r},\uv_{n-1},\uv_{r+1},\dots,\uv_{n-1}),
\]
thus $\uuv^{(i+1,0)} = (\tuv_{i+1},\dots,\tuv_{i+r},\uv_{r+1},\dots,\uv_{n-1},
\tuv_{i+1})$ since $k_{i+1,0}=1$.  Again this proves \eqref{app:claim:v(i,0)}
by induction on $i$.

By \eqref{app:eq:X:0}, we have $a_1^{(0,0)}=a_0\ge \log(X_0)-c\ge \kappa$.
Together with \eqref{app:eq:delta,c} and the fact that $\ell_{i,j}=n$ for all pairs
$(i,j)$ with $i\ge 0$ and $0\le j\le \nu(i)$, this shows that
the $n$-system $\uP$ and its associated coherent sequence of directions $(\uuv^{(i,j)})$
satisfy the hypotheses of Theorem \ref{param:thm}.  For the corresponding
constant $c_2$, we find
\begin{equation}
 \label{app:thm:eq:c2}
 c_2=n\log(32)+\log(n!)\le n\log(32n)\le nc.
\end{equation}
Let $\uu$ denote the unit vector of $\bR^n$, and
$\uux^{(i,j)}=(\ux_1^{(i,j)},\dots,\ux_n^{(i,j)})$, the generic element
of the coherent sequence of bases of $\bZ^n$ provided by this theorem.  We set
\begin{equation}
 \ux_i = \ux_1^{(i,0)} \quad \text{for each $i\ge 0$.}
\label{app:thm:eq:xi}
\end{equation}
It remains to show that the point $\uu$ and the sequence $(\ux_i)_{i\ge 0}$
have the required properties.  By \eqref{param:thm:eq5}, we first note,
 using \eqref{app:thm:eq:c2}, that
\begin{equation}
 \label{app:thm:eq:LP}
 \norm{\uL_\uu(q)-\uP(q)}_\infty\le  c_2\le nc
\quad\text{for each $q\ge q_{0,0}$.}
\end{equation}
Since $P_1(q)$ tends to infinity with $q$, the same applies to $L_{\uu,1}(q)$ and
so the coordinates of $\uu$ are linearly independent over $\bQ$.

Fix an index $i\ge 0$.
By \eqref{app:claim:v(i,0)}, we have $\tuv_1^{(i,0)}=\tuv_i$. Thus the estimate
\eqref{param:thm:eq1} applied to the first element of the basis $\uux^{(i,0)}$ yields
\[
  \dist(\ux_i,\tuv_i)=\dist(\ux_1^{(i,0)},\tuv_1^{(i,0)})\le \delta,
\]
as requested in condition 1).  Similarly, since $a_1^{(i,0)}=a_i$ and $k_{i,0}=1$,
the estimates \eqref{param:thm:eq2} and \eqref{param:thm:eq3} provide
respectively
\begin{equation}
  \big|\log\norm{\ux_i}-a_i\big|\le \log(2)
  \et
  \big|\log|\ux_i\cdot\uu| - a_i+ q_{i,0} \big| \le c_2
 \label{app:thm:eq:Lx}
\end{equation}
because, using \eqref{app:thm:eq:ba} and \eqref{app:thm:eq:c2}, we find
\[
 q_{i,1}-q_{i,0}=b_i+c-a_i\ge b_i+c-a_{i+r-1}\ge (n+1)c\ge \log(2)+c_2.
\]
By \eqref{app:eq:m*}--\eqref{app:def:ab}, we also observe that,
under the componentwise ordering on $\bR^n$, we have
\begin{equation}
\left.
\begin{array}{rcl}
 -2m\uc\ \le &\ua^{(i,0)} - \tuP(q_i) &\le\ 0,\\
 -2m\uc\ \le &\ua^{(i,1)}- \tuP(s_i) &\le\ 0,\\
 -2m\uc\ \le &\ua^{(i,\sigma(i)m)} - \tuP(t_i) &\le\ 0,
\end{array}
\right\}
\quad\text{where}\quad \uc=(c,\dots,c).
 \label{app:eq:qst1}
\end{equation}
For example, we obtain the last estimate by observing that
\[
 \begin{aligned}
 \ua^{(i,\sigma(i)m)} - \tuP(t_i)
 = (a_{i+1}-\log(X_{i+1}),&\dots,a_{i+r-1}-\log(X_{i+r-1}),\\
     &a_{i+r}-mc-\log(X_{i+r}),\dots,a_{i+r}-\log(X_{i+r})).
 \end{aligned}
\]
Taking the sum of the coordinates of each term in the inequalities \eqref{app:eq:qst1},
we deduce that
\begin{equation}
 \label{app:eq:qst2}
 -2mnc\ \le\ q_{i,0}-q_i,\ q_{i,1}-s_i,\ q_{i,\sigma(i)m}-t_i \ \le\ 0.
\end{equation}
In particular, we have $|q_i-q_{i,0}|\le 2mnc$.  As \eqref{app:eq:ab} gives
$|a_i-\log(X_i)|\le mc$, we deduce from \eqref{app:thm:eq:Lx} that
\begin{align*}
  &\big|\log\norm{\ux_i}-\log(X_i)\big|
  \le \log(2)+mc\le nc,\\
  &\big|\log|\ux_i\cdot\uu| - \log(X_i)+ q_i \big|
   \le nc+mc+2mnc \le 4mnc.
\end{align*}
Thus condition 2) holds.

By \eqref{app:eq:qst2}, we also have $q_{0,0}\le q_0$.  So, in view of
\eqref{app:thm:eq:LP}, it suffices to show that
\begin{equation}
 \label{app:claim:P}
 \norm{\uP(q)-\tuP(q)}_\infty \le 4mnc \quad\text{for each $q\ge q_0$,}
\end{equation}
in order to prove condition 3) and thus complete the whole proof.

To do this, fix a real number $q$ with $q\ge q_0$.   Since $q_0\ge q_{0,0}$, there is
an integer $i\ge 0$ for which $q_{i,0}\le q < q_{i+1,0}$.   To prove
\eqref{app:claim:P} for this value of $q$, it suffices to show that
\begin{equation}
 \label{app:claim:bis}
 \uP(q) \le \tuP(q^*) \quad \text{where $q^*=q+2mnc$.}
\end{equation}
Indeed, if we admit this inequality, then all the coordinates of the $n$-tuple
$\tuP(q^*)-\uP(q)$ are non-negative.  Since they sum to $q^*-q=2mnc$,
these coordinates are at most $2mnc$, and so $\norm{\tuP(q^*)-\uP(q)}_\infty
\le 2mnc$.  Since each component of $\tuP$ is continuous and piecewise linear with
slope at most $1$, we also have $\norm{\tuP(q^*)-\tuP(q)}_\infty
\le 2mnc$, and thus $\norm{\uP(q)-\tuP(q)}_\infty \le 4mnc$.  To prove
\eqref{app:claim:bis}, we distinguish three cases.

\textbf{Case 1.} Suppose first that $q_{i,0}\le q< q_{i,1}$.  If $q+q_i-q_{i,0}\le s_i$,
we find
\[
 \uP(q)=\Phi_n(\ua^{(i,0)}+(q-q_{i,0})\ue_1)
 \le \Phi_n(\tuP(q_i)+(q-q_{i,0})\ue_1)
  = \tuP(q+q_i-q_{i,0}),
\]
where the inequality in the middle uses the fact that $\Phi_n$ is order preserving
together with the first inequality in \eqref{app:eq:qst1}.   If instead
$s_i < q+q_i-q_{i,0}$, then using the second inequality in \eqref{app:eq:qst1}, we find
\[
 \uP(q)\le \uP(q_{i,1})=\ua^{(i,1)} \le \tuP(s_i) \le \tuP(q+q_i-q_{i,0}).
\]
Since \eqref{app:eq:qst2} gives $q_i-q_{i,0}\le 2mnc$, we deduce that
$\uP(q)\le \tuP(q^*)$ in both instances.

From now on, we may therefore assume that $q\ge q_{i,1}$.  Then, for each
$j=1,\dots,r-1$, we find that
\[
 P_j(q)=a_{i+j}\le \log(X_{i+j})=\tP_j(s_i)\le \tP_j(q^*),
\]
where the last inequality uses $s_i\le q_{i,1}+2mnc\le q^*$ coming from
\eqref{app:eq:qst2}.

\textbf{Case 2.} Suppose that $q_{i,1}\le q\le q_{i,\sigma(i)m}$.  If $q^*\le t_i$,
then, for $j=r,\dots,n$, we find that
\[
 P_j(q) \le P_n(q)
  \le b_i+c+\frac{q-q_{i,1}}{m+1}
  \le \log(Y_i)+\frac{q^*-s_i}{m+1}
  = \tP_j(q^*),
\]
because $b_i+c\le \log(Y_i)$ by \eqref{app:eq:ab}, and $q-q_{i,1}\le q^*-s_i$
by \eqref{app:eq:qst2}.  If instead $q^*>t_i$, then, for the same values
of $j$, we find
\[
 P_j(q)\le P_n(q)
  \le P_n(q_{i,\sigma(i)m})=a_j^{(i,\sigma(i)m)}
  \le \tP_j(t_i) \le \tP_j(q^*),
\]
using the third inequality in \eqref{app:eq:qst1}.  Thus, \eqref{app:claim:bis} holds in
both instances.

\textbf{Case 3.} Suppose that $q_{i,\sigma(i)m}\le q$.  Using \eqref{app:eq:qst2}, we find
$0\le q-q_{i,\sigma(i)m}\le q^*-t_i$, thus $t_i\le q^*$ and so
\[
 P_r(q)\le a_{i+r}\le \log(X_{i+r}) = \tP_r(t_i) \le \tP_r(q^*).
\]
If $q^*\le q_{i+1}$, we also obtain, for $j=r+1,\dots,n$,
\[
 P_j(q)\le P_n(q) \le a_{i+r}+\frac{q-q_{i,\sigma(i)m}}{m}
  \le \log(X_{i+r}) +\frac{q^*-t_i}{m}= \tP_j(q^*).
\]
If instead $q^*>q_{i+1}$, then, for the same values of $j$, we find
\[
 P_j(q)\le P_n(q) \le P_n(q_{i+1,0}) = b_{i+1}
  \le \log(Y_{i+1}) = \tP_j(q_{i+1}) \le \tP_j(q^*).
\]
Thus, \eqref{app:claim:bis} holds in that case also.
\end{proof}

\subsection*{Proof of Theorem \ref{intro:thm1}, part 2)}
Let $V$ be a subspace of $\bR^n$ of dimension $m+1$, and let
$\psi\colon[1,\infty)\to(0,\infty)$ be an unbounded monotonically increasing
function.  Since $V$ has the same dimension as
\[
 V_0=\langle\ue_1+\cdots+\ue_r,\ue_{r+1},\dots,\ue_n\rangle_\bR,
\]
there is an isometry $T$ of $\bR^n$ which maps $V_0$ to $V$.   Setting
$\uv_j=T(\ue_j)$ for each $j=1,\dots,n$, we obtain an orthonormal basis
$(\uv_1,\dots,\uv_n)$ of $\bR^n$ such that
\[
 V=\langle\uv_1+\cdots+\uv_r,\uv_{r+1},\dots,\uv_n\rangle_\bR.
\]
We apply the previous theorem to this choice of basis $(\uv_1,\dots,\uv_n)$
for
\begin{equation}
\label{proof:eq:delta,c}
 \delta=1/\max\{4n, 24r\}  \et  c=\log(8/\delta),
\end{equation}
so that \eqref{app:eq:delta,c} holds, and for sequences $(X_i)_{i\ge0}$ and
$(Y_i)_{i\ge 0}$ satisfying \eqref{app:eq:X:0}--\eqref{app:eq:XY}
as well as
\begin{equation}
 \label{proof:eq:Y,psi}
 \log(Y_i)=(\rho/m)\log(X_{i+r-1})
 \et
 X_iX_{i+1}\cdots X_{i+r-1}\le \psi(X_{i+r}/X_i)
\end{equation}
for each $i\ge 0$, where $\rho=\rho_m$ is given by \eqref{intro:thm1:eq:rho}.  This
is possible since $\rho>m$.  We claim that the unit vector $\uu\in\bR^n$ provided by
Theorem \ref{app:thm} has the property stated in part 2) of Theorem \ref{intro:thm1}.
Since its coordinates are linearly independent
over $\bQ$, this amounts to showing that any non-zero point $\ux$ of $\bZ^n$
of sufficiently large norm with $\dist(\ux,V)\le\delta$ satisfies
\begin{equation}
 |\ux\cdot\uu| > \psi(\norm{\ux})^{-1}\norm{\ux}^{-\rho}.
 \label{app:proof:xu}
\end{equation}
To prove this, we use properties 1)--3) of the associated sequences
$(\tuv_i)_{i\ge 0}$ and $(\ux_i)_{i\ge 0}$ in Theorem  \ref{app:thm}.

Let $\ux$ be a non-zero point of $\bZ^n$ with $\dist(\ux,V)\le\delta$.  Assuming,
as we may, that $\norm{\ux}$ is large enough, there exists an integer $i\ge 0$ such that
\[
 \log(Y_i) \le \log\norm{\ux} + 15c' \le \log(Y_{i+1})  \quad\text{where}\quad c'=mnc,
\]
and so there is a unique value of $q\in[s_i,q_{i+1}]$ for which
\[
 \log\norm{\ux}=\tP_{r+1}(q)-15c'.
\]
For such $q$, we note the following useful formula
\begin{equation}
 \label{proof:eq:P}
 \tP_{r+1}(q)=\tP_r(q)+\max\{0,q-t_i\}/m.
\end{equation}
In the computations below, we assume that $\norm{\ux}$ is large enough so that
for example we have $\log(Y_i/X_{i+r-1})\ge 9c'$.  To simplify
the exposition, we simply put a star on the inequalities that require such additional
assumptions.   We consider two cases.

\noindent
\textbf{Case 1:} Suppose that $L_\uu(\ux,q)=\log\norm{\ux}$.

We first note that $(\tuv_{i+1},\dots,\tuv_{i+r},\uv_{r+1},\dots,\uv_n)$ is an
orthonormal basis of $\bR^n$ and that
\[
 V=\langle \tuv_{i+1}+\cdots+\tuv_{i+r}, \uv_{r+1}, \dots, \uv_n \rangle_\bR,
\]
because $(\tuv_{i+1},\dots,\tuv_{i+r})$ is a permutation of $(\uv_1,\dots,\uv_r)$.
Since property 1) gives
\[
 \dist(\ux_{i+j},\tuv_{i+j}) \le \delta\le 1/(24r) \quad \text{for $j=1,\dots,r$,}
\]
by the choice of $\delta$ in \eqref{proof:eq:delta,c}, it follows from
Lemma \ref{dist:lemma2} that $(\ux_{i+1},\dots,\ux_{i+r})$ is a linearly
independent $r$-tuple of points of $\bZ^n$.  We claim that if $\norm{\ux}$
is large enough, we also have
\begin{equation}
 \label{proof:claim}
 \ux \in \langle \ux_{i+1},\dots,\ux_{i+r} \rangle_\bR,
\end{equation}
and thus Lemma \ref{dist:lemma3} applies.
To prove this, we look more closely at the trajectories of the points
$\ux,\ux_{i+1},\dots,\ux_{i+r}$.  For each $j\ge 1$, we have that
$q\le q_{i+1}\le q_{i+j}$, and so property 2) in Theorem  \ref{app:thm} yields
\[
 L_\uu(\ux_{i+j},q)
  \le \max\{\log\norm{\ux_{i+j}}, q_{i+j}+\log|\ux_{i+j}\cdot\uu|\}
  \le \log(X_{i+j}) + 4c'.
\]
We distinguish two sub-cases depending on the value of $q$.

a) Suppose first that $s_i\le q < t_i+10mc'$.  Then, the above estimates give
\[
 \max_{1\le j<r} L_\uu(\ux_{i+j},q)
  \le \log(X_{i+r-1})+4c'
   <^* \log(Y_i)-5c' \le \tP_r(q)-5c',
\]
while \eqref{proof:eq:P} yields
\[
 L_\uu(\ux,q)=\log\norm{\ux}=\tP_{r+1}(q)-15c' < \tP_r(q)-5c'.
\]
As $\tP_r(q)\le L_{u,r}(q)+5c'$ by property 3), this means that
 \[
\max\{L_\uu(\ux,q),L_\uu(\ux_{i+1},q),\dots,L_\uu(\ux_{i+r-1},q)\} <L_{\uu,r}(q).
\]
Thus, the $r$ points $\ux,\ux_{i+1},\dots,\ux_{i+r-1}\in \bZ^n$ are linearly
dependent and so the point $\ux$ belongs to $\langle \ux_{i+1},\dots,\ux_{i+r-1} \rangle_\bR$
which is stronger than our claim \eqref{proof:claim}.

b) Suppose now that  $t_i+10mc'\le q \le q_{i+1}$.  Then, using \eqref{proof:eq:P},
we find
\[
 \max_{1\le j\le r} L_\uu(\ux_{i+j},q)
  \le \log(X_{i+r})+4c'
  \le \tP_{r+1}(q)-6c'.
\]
As $\tP_{r+1}(q)\le L_{\uu,r+1}(q)+5c'$ and $L_\uu(\ux,q)=\log\norm{\ux}
=\tP_{r+1}(q)-15c'$, this means that
\[
\max\{L_\uu(\ux,q),L_\uu(\ux_{i+1},q),\dots,L_\uu(\ux_{i+r},q)\} \le L_{\uu,r+1}(q) - c'.
\]
So, the $r+1$ points $\ux,\ux_{i+1},\dots,\ux_{i+r}\in \bZ^n$ are linearly
dependent and thus \eqref{proof:claim} holds again.

Applying Lemma \ref{dist:lemma3}, we can therefore write
\[
 \ux=a_1\ux_{i+1}+\cdots+a_r\ux_{i+r}
\]
with coefficients $a_1,\dots,a_r\in\bR$ that satisfy
\begin{equation}
 \label{proof:eq:a}
 \frac{1}{2\sqrt{r}}\le \frac{\norm{a_j\ux_{i+j}}}{\norm{\ux}} \le \frac{2}{\sqrt{r}}
 \quad
 \text{for $j=1,\dots,r$.}
\end{equation}
This yields the lower bound
\[
 \abs{\ux\cdot\uu}
 \ge \abs{a_1\ux_{i+1}\cdot\uu}-\sum_{j=2}^r\abs{a_j\ux_{i+j}\cdot\uu}
 \ge \frac{\norm{\ux}}{2\sqrt{r}}
      \left( \frac{\abs{\ux_{i+1}\cdot\uu}}{\norm{\ux_{i+1}}}
           - 4\sum_{j=2}^r\frac{\abs{\ux_{i+j}\cdot\uu}}{\norm{\ux_{i+j}}} \right)
\]
By property 2), we also have
\[
 \abs{\log\frac{\abs{\ux_j\cdot\uu}}{\norm{\ux_j}}+q_j} \le 5c'
 \quad\text{for each $j\ge 0$.}
\]
Thus the previous estimate yields
\[
 \begin{aligned}
 \abs{\ux\cdot\uu}
 &\ge \frac{\norm{\ux}}{2\sqrt{r}}
      \big( \exp(-q_{i+1}-5c') - 4r\exp(-q_{i+2}+5c') \big)\\
 &>^* \frac{\norm{\ux}}{c''\exp(q_{i+1})}
   = \frac{\norm{\ux}}{c''X_{i+1}\cdots X_{i+r} Y_{i+1}^m}
   = \frac{\norm{\ux}}{c''X_{i+1}\cdots X_{i+r-1} X_{i+r}^{\rho+1}}
 \end{aligned}
\]
where $c''=4\sqrt{r}\exp(5c')$.  The inequalities \eqref{proof:eq:a} also
imply that $a_r$ is non-zero.
This means that $\ux\notin\langle \ux_{i+1},\dots,\ux_{i+r-1} \rangle_\bR$ and
thus rules out the sub-case a) considered above.  So we have $q\ge t_i+10mc'$,
and using \eqref{proof:eq:P} we find that
\[
 \norm{\ux}=\exp(\tP_{r+1}(q)-15c') \ge X_{i+r}\exp(-5c').
\]
Combining the last two estimates and using \eqref{proof:eq:Y,psi}, we conclude as
announced that
\[
  \abs{\ux\cdot\uu}
 >^* \frac{\norm{\ux}^{-\rho}}{X_iX_{i+1}\cdots X_{i+r-1} }
 \ge \frac{\norm{\ux}^{-\rho}}{\psi(X_{i+r}/X_i)}
 \ge^* \frac{\norm{\ux}^{-\rho}}{\psi(\norm{\ux})}.
\]

\noindent
\textbf{Case 2:} Suppose instead that $L_\uu(\ux,q)>\log\norm{\ux}$.

By definition, this means that
\begin{equation}
\label{proof:case2:eq1}
 \log\abs{\ux\cdot\uu} > \log\norm{\ux}-q.
\end{equation}
Since the coordinates of $\tuP(q)$ sum to $q$, we have
\[
 q = \log(X_{i+1})+\cdots+\log(X_{i+r-1})
    + \begin{cases}
        (m+1)\tP_{r+1}(q) &\text{if $s_i\le q\le t_i$,}\\[5pt]
        \log(X_{i+r}) + m\tP_{r+1}(q) &\text{if $t_i\le q\le q_{i+1}$.}
       \end{cases}
\]
As $\tP_{r+1}(q)\ge \log(X_{i+r})$ when $q\ge t_i$, this implies in all cases that
\[
 q \le \log(X_{i+1})+\cdots+\log(X_{i+r-1})+(m+1)\tP_{r+1}(q).
\]
We also have
\begin{equation}
\label{proof:case2:eq2}
 \log\norm{\ux}+15c'=\tP_{r+1}(q) \ge \log(Y_i) =(\rho/m)\log(X_{i+r-1}).
\end{equation}
Using this to eliminate $\log(X_{i+r-1})$ and $\tP_{r+1}(q)$ from
the upper bound for $q$, we obtain
\[
 q \le \log(X_{i+1})+\cdots+\log(X_{i+r-2})+(m+1+m/\rho)(\log\norm{\ux}+15c').
\]
Now, we use the exact value of $\rho$.  Since $m+m/\rho=\rho$, we deduce that
\[
 q \le^* \log(X_{i})+\cdots+\log(X_{i+r-2})+(\rho+1)\log\norm{\ux}.
\]
Together with \eqref{proof:case2:eq1}, this yields
\[
 \abs{\ux\cdot\uu}
  > \frac{\norm{\ux}}{\exp(q)}
  \ge \frac{\norm{\ux}^{-\rho}}{X_i\cdots X_{i+r-2}}
  \ge \frac{\norm{\ux}^{-\rho}}{\psi(X_{i+r-1})}.
\]
Finally, since $\rho/m>1$, we obtain from \eqref{proof:case2:eq2} that
$X_{i+r-1}\le^* \norm{\ux}$ and thus \eqref{app:proof:xu} holds
if $\norm{\ux}$ is large enough.

%
%
\section{A simplification of Thurnheer's argument}
\label{sec:simplif}

As mentioned in the introduction, the first part of Theorem \ref{intro:thm1} is
a result of Thurnheer
when $n=m+2$.  Our goal in this last section is to provide a simplification
of his argument along the lines of \cite{DS1968}.  The following statement is a slight
generalization of \cite[Theorem 1 (b)]{Th1990}.

\begin{theorem}[Thurnheer]
\label{simplif:thm}
Let $m$ be a positive integer, let $n=m+2$, let $\uu$ be a point of
$\bR^n$ whose coordinates are linearly independent over $\bQ$, and
let $\ud\in\bR^n$.  Then, for any $\delta,\epsilon\in(0,1)$, there is a
non-zero point $\ux$ of $\bZ^n$ with
\[
 \abs{\ux\cdot\ud}\le \delta\norm{\ux}
 \et
 \abs{\ux\cdot\uu}\le \epsilon\norm{\ux}^{-\rho}
\]
where $\rho$ is as in Theorem \ref{intro:thm1}.
\end{theorem}

For the proof, we assume that $(\uu,\ud)$ is an
orthonormal pair of vectors in $\bR^n$ because, if the conclusion holds for such
a pair, then it also holds for any pair $(a\uu,b\uu+c\ud)$ with $a,b,c\in\bR$ and
$a\neq 0$, and this covers the general case.
From there, we proceed by contradiction.  We suppose that there exist
numbers $0<\delta,\epsilon<1$
such that any non-zero point $\ux\in\bZ^n$ with $\abs{\ux\cdot\uu}\le
\epsilon\norm{\ux}^{-\rho}$ satisfies $\abs{\ux\cdot\ud} > \delta\norm{\ux}$.
To derive a contradiction, we define a norm $\norm{\ }'$ on $\bR^n$ through
the formula
\[
 \norm{\ux}'
  = \max\big\{ \abs{\ux\cdot\ud}, (\delta/4)\norm{\ux}\big\}
\]
for each $\ux\in\bR^n$.  Then, our hypothesis is that, for any non-zero point
$\ux$ of $\bZ^n$, we have
\begin{equation}
\label{simplif:hyp}
 \abs{\ux\cdot\uu} \le \epsilon\norm{\ux}^{-\rho}
 \quad\Longrightarrow\quad
 \norm{\ux}'=\abs{\ux\cdot\ud} > \delta\norm{\ux}.
\end{equation}

Extending an argument of Schmidt in \cite[Lemma 1]{Sc1976} for the case $n=3$,
Thurnheer obtained the following general estimate in \cite[Part II (iii)]{Th1990}.

\begin{lemma}[Schmidt and Thurnheer]
\label{simplif:lemma:convex}
There are constants $c_3$ and $X_0$ depending only on $n$ and $\epsilon$
with the property that, for each real number $X$ with $X\ge X_0$, there exists a
non-zero point $\ux$ of\/ $\bZ^n$ with $\norm{\ux}'\le X$ and $\abs{\ux\cdot\uu}
\le c_3X^{-\rho-1}$.
\end{lemma}

\begin{proof}
We first complete the pair $(\uu,\ud)$ to an orthonormal basis
$(\uw_1,\dots,\uw_m,\uu,\ud)$ of $\bR^n$.  Then, for any positive real
number $Y$, the set of points $\ux\in\bR^n$ satisfying
\begin{equation}
\label{simplif:lemma:convex:eq}
 \max_{1\le j\le m}\abs{\ux\cdot\uw_j}\le n^{-1/2}Y,
 \quad
 \abs{\ux\cdot\uu}\le \epsilon Y^{-\rho}
 \et
 \abs{\ux\cdot\ud}\le n^{m/2}\epsilon^{-1} Y^{\rho-m}
\end{equation}
is a compact symmetric convex body of $\bR^n$ of volume $2^n$ and so, by
Minkowski's first convex body theorem, it contains a non-zero point $\ux$ of
$\bZ^n$.  Since $m<\rho<m+1$ by \eqref{intro:ineq:rho}, the above inequalities yield
$\abs{\ux\cdot\uu}\le n^{-1/2}Y$ and $\abs{\ux\cdot\ud}\le n^{-1/2}Y$
if $Y$ is large enough in terms of $n$ and $\epsilon$.  Then we find
\[
 \norm{\ux}
 = \big( \abs{\ux\cdot\uw_1}^2+\cdots+\abs{\ux\cdot\uw_m}^2
      + \abs{\ux\cdot\uu}^2 + \abs{\ux\cdot\ud}^2 \big)^{1/2} \le Y
\]
and so $\abs{\ux\cdot\uu} \le \epsilon\norm{\ux}^{-\rho}$ by the middle
inequality in \eqref{simplif:lemma:convex:eq}.  According to our hypothesis \eqref{simplif:hyp}
and the formula for $\rho$ in \eqref{intro:ineq:rho},  this implies that
\[
 \norm{\ux}'=\abs{\ux\cdot\ud}\le c_2 Y^{\rho-m} = c_2 Y^{\rho/(\rho+1)},
\]
where $c_2=n^{m/2}\epsilon^{-1}$.  If $X$ is large enough, and if $Y$ is chosen
so that $X=c_2 Y^{\rho/(\rho+1)}$, the point $\ux$ constructed above
satisfies $\norm{\ux}'\le X$ and $\abs{\ux\cdot\uu}\le
\epsilon Y^{-\rho} = c_3X^{-\rho-1}$ where $c_3=\epsilon c_2^{\rho+1}$.
\end{proof}

Since the coordinates of $\uu$ are linearly independent over $\bQ$, the
scalar product with $\uu$ defines an injective map from $\bZ^n$ to $\bR$.
Thus, for each $X\in [1,\infty)$, there is, up to multiplication by $\pm 1$, a
unique non-zero point $\ux\in\bZ^n$ with $\norm{\ux}'\le X$ for which
$\abs{\ux\cdot\uu}$ is minimal.  We order these pairs $\pm \ux$
by increasing norm $\norm{\ux}'$ and, for each integer $i\ge 1$,
we choose a representative $\ux_i$ of the $i$-th pair for which
$\ux_i\cdot\ud\ge 0$ (unique unless $\ux_i\cdot\ud=0$).  Then, each $\ux_i$
is a primitive point of $\bZ^n$ and each pair $(\ux_i,\ux_{i+1})$ is linearly
independent.  We also set
\[
 X_i=\norm{\ux_i}' \et L_i=\abs{\ux_i\cdot\uu}
 \quad \text{for each $i\ge 1$.}
\]
By construction, the sequence $(X_i)_{i\ge 1}$ is strictly increasing, while
$(L_i)_{i\ge 1}$ is strictly decreasing.  Moreover, by Lemma \ref{simplif:lemma:convex},
we have $\abs{\ux_i\cdot\uu} \le c_3 X^{-\rho-1}$ for each $i\ge 1$ with
$X_i\ge X_0$ and each $X\in[X_0,X_{i+1})$.

Let $i_0$ denote the smallest integer $i\ge 1$ for which
$X_i\ge X_0$ and $c_3X_i^{-1}\le \epsilon(4/\delta)^{-\rho}$.  By the above,
we deduce that, for each integer $i$ with $i\ge i_0$, we have
\begin{equation}
\label{simplif:eq:Li}
 L_i\le c_3 X_{i+1}^{-\rho-1},
\end{equation}
and also $\abs{\ux_i\cdot\uu}  \le c_3X_i^{-\rho-1}\le \epsilon ((4/\delta)X_i)^{-\rho}
\le \epsilon \norm{\ux_i}^{-\rho}$ since $\norm{\ux_i}\le (4/\delta)X_i$.   Then our
hypothesis \eqref{simplif:hyp} together with the condition $\ux_i\cdot\ud\ge 0$ gives
\begin{equation}
\label{simplif:eq:Xi}
 X_i=\norm{\ux_i}'=\ux_i\cdot\ud > \delta\norm{\ux_i} \quad(i\ge i_0).
\end{equation}
From this, we deduce two consequences by adapting the
arguments of Davenport and Schmidt in \cite[Lemmas 1 and 2]{DS1968}.
The first lemma below is also implicit in \cite[Part II (v)]{Th1990}.

\begin{lemma}
\label{simplif:lemma:T2}
For each $i\ge i_0$, the scalar products $\ux_i\cdot\uu$ and $\ux_{i+1}\cdot\uu$
have opposite signs.
\end{lemma}

\begin{proof}
Suppose on the contrary that they have the same sign for some $i\ge i_0$.
Then the point $\ux=\ux_{i+1}-\ux_i\in\bZ^n$ is non-zero and satisfies
\[
\abs{\ux\cdot\uu} = L_i - L_{i+1} < L_i.
\]
By construction, this implies that $\norm{\ux}' \ge X_{i+1}$.  However,
using the relations \eqref{simplif:eq:Xi}, we find that
$\abs{\ux\cdot\ud} = X_{i+1}-X_i < X_{i+1}$ and
\[
 \frac{\delta}{4}\norm{\ux}
    \le \frac{\delta}{4}\norm{\ux_{i+1}}+ \frac{\delta}{4}\norm{\ux_i}
    \le \frac{X_{i+1}}{4} + \frac{X_i}{4} < X_{i+1},
\]
which imply that $\norm{\ux}' <X_{i+1}$, a contradiction.
\end{proof}

\begin{lemma}
\label{simplif:lemma:lindep}
For each $i>i_0$, the ratios $L_{i-1}/L_{i}$ and $X_{i+1}/X_{i}$
have the same integer part and, for this integer $t_i$, we have
\[
 \ux_{i+1}=t_i\ux_i+\ux_{i-1}.
\]
\end{lemma}

\begin{proof}
Let $s$ and $t$ denote respectively the integer parts of $L_{i-1}/L_{i}$
and $X_{i+1}/X_{i}$, for a choice of index $i > i_0$.   Our first goal is to show
that $s=t$.  To this end, we note, using \eqref{simplif:eq:Xi},
that the non-zero integer point $\uy=\ux_{i+1}-t\ux_{i} \in \bZ^n$ satisfies
\begin{equation}
\label{simplif:lemma:lindep:eq0}
 \norm{\uy}
   \le \norm{\ux_{i+1}} +\frac{X_{i+1}}{X_i}\norm{\ux_i}
   \le \frac{X_{i+1}}{\delta} +\frac{X_{i+1}}{X_i}\frac{X_i}{\delta}
   \le \frac{2}{\delta}X_{i+1}.
\end{equation}
Since $\ux_{i}\cdot\uu$ and $\ux_{i+1}\cdot\uu$ have opposite signs
by Lemma \ref{simplif:lemma:T2}, we also find
\[
 \abs{\uy\cdot\uu} = L_{i+1}+tL_{i}
  \le 2 \frac{X_{i+1}}{X_{i}}L_{i}
  \le \frac{2c_3}{X_i}X_{i+1}^{-\rho},
\]
where the last inequality uses \eqref{simplif:eq:Li}.  We deduce that
$\abs{\uy\cdot\uu}\le \epsilon\norm{\uy}^{-\rho}$ because, as $i\ge i_0$, we
have $2c_3/X_i \le 2\epsilon(4/\delta)^{-\rho}\le \epsilon(2/\delta)^{-\rho}$.
Then, hypothesis \eqref{simplif:hyp} combined with  \eqref{simplif:eq:Xi} yields
\begin{equation*}
\label{simplif:lemma:lindep:eq1}
 \norm{\uy}'=\abs{\uy\cdot\ud}=X_{i+1}-tX_{i} < X_{i}.
\end{equation*}
By the minimality of $\ux_i$, this implies that $\abs{\uy\cdot\uu}\ge L_{i-1}$.  Since
$\abs{\uy\cdot\uu}=L_{i+1}+tL_{i}$, we conclude that $L_{i-1}<(t+1)L_{i}$ and so
$s\le t$.

Similarly, Lemma \ref{simplif:lemma:T2} implies that the non-zero point $\uz=s\ux_i+\ux_{i-1}
\in \bZ^n$ satisfies
\[
 \abs{\uz\cdot\uu} = L_{i-1}-sL_{i} < L_{i}.
\]
So, we must have $\norm{\uz}' \ge X_{i+1}$.  This yields
\[
 t\le \frac{X_{i+1}}{X_{i}}
  \le \frac{\norm{\uz}'}{X_{i}}
  \le \frac{sX_i+X_{i-1}}{X_{i}} < s + 1,
\]
thus $t\le s$, and so $s=t$.

Finally, consider the point $\ux=\ux_{i+1}-t\ux_{i}-\ux_{i-1}\in\bZ^n$. Using
Lemma \ref{simplif:lemma:T2} and the equality $s=t$, we find
\begin{align*}
 \abs{\ux\cdot\uu}
  &=\abs{L_{i+1}+tL_i-L_{i-1}}=\abs{L_{i+1}-(L_{i-1}-sL_i)} < L_i,\\
 \abs{\ux\cdot\ud}
  &=\abs{(X_{i+1}-tX_i)-X_{i-1}} < X_i.
\end{align*}
Since $\ux=\uy-\ux_{i-1}$, estimates \eqref{simplif:eq:Xi} and
\eqref{simplif:lemma:lindep:eq0} yield
\[
 \frac{\delta}{4}\norm{\ux}
  \le \frac{\delta}{4}\norm{\uy} + \frac{\delta}{4}\norm{\ux_{i-1}}
  \le \frac{1}{2}X_{i+1} + \frac{1}{4}X_{i-1}
     < X_{i+1},
\]
thus $\norm{\ux}'< X_{i+1}$.  Since $\abs{\ux\cdot\uu}<L_i$, this means
that $\ux=0$, and so $\ux_{i+1}=t\ux_i+\ux_{i-1}$.
\end{proof}

The last lemma leads to a contradiction by arguing as in \cite[Lemma 3]{DS1967}.
Indeed, Lemma \ref{simplif:lemma:lindep} shows that the subgroup of $\bR^n$
spanned by $\ux_i$ and $\ux_{i+1}$ is independent of $i$ for $i\ge i_0$,
thus $\ux_i\wedge\ux_{i+1} = \pm \ux_{j}\wedge\ux_{j+1}$ for any choice
of integers $i,j$ with $i_0\le i<j$.  Contracting these bi-vectors with $\uu$ and
then taking norms, we obtain
\begin{align*}
 \norm{(\ux_i\cdot\uu)\ux_{i+1}-(\ux_{i+1}\cdot\uu)\ux_i}
 &= \norm{(\ux_j\cdot\uu)\ux_{j+1}-(\ux_{j+1}\cdot\uu)\ux_j}\\
 &\le \frac{1}{\delta}X_{j+1}L_j+\frac{1}{\delta}X_{j}L_{j+1}
  \le \frac{2}{\delta}X_{j+1}L_j
  \le \frac{2c_3}{\delta}X_{j+1}^{-\rho}.
\end{align*}
Letting $j$ go to infinity for fixed $i\ge i_0$, we deduce that $(\ux_i\cdot\uu)\ux_{i+1} =
(\ux_{i+1}\cdot\uu)\ux_i$.  However this is impossible since $\ux_i$
and $\ux_{i+1}$ are linearly independent over $\bR$ and $\ux_i\cdot\uu
\neq 0$. This contradiction completes the proof of Theorem \ref{simplif:thm}.

\section*{Acknowledgement}
The authors thanks Michel Laurent for the helpful reference to Erd\"os's paper \cite{Er1958}.
The work of both authors was partially supported by an NSERC discovery grant.

%
%

\end{document}